\documentclass[a4paper,11pt,         twoside
               ]{article}

\usepackage[top=1in, bottom=1in, left=1.0in, right=1.0in]{geometry}
\usepackage{indentfirst}
\usepackage{xcolor}
\usepackage{subcaption}
\usepackage{appendix}
\usepackage{soul}
\usepackage{amsmath}
\allowdisplaybreaks

\usepackage[perpage,symbol]{footmisc}
\usepackage[sf]{titlesec}            
\usepackage{comment}
\date{}
\usepackage{titletoc}                
\usepackage{fancyhdr}                
\usepackage{type1cm}                 
\usepackage{indentfirst}             
\usepackage{makeidx}                 
\usepackage{textcomp}                
\usepackage{layouts}                 
\usepackage{bbding}                  
\usepackage{cite}                    
\usepackage{color,xcolor}            
\usepackage{listings}                

\usepackage{times}     
\usepackage{palatino} 
\usepackage{newcent}  
\usepackage{bookman}  
\usepackage{latexsym}
\usepackage{amsmath}                 
\usepackage{amssymb}                 
\usepackage{amsbsy}
\usepackage{amsthm}
\usepackage{amsfonts}
\usepackage{mathrsfs}                
\usepackage{bm}                      
\usepackage{relsize}                 
\usepackage{booktabs}
\usepackage{float}
\usepackage{epsfig}                  
\usepackage{hyperref}
\newtheorem{theorem}{Theorem}[section]
\newtheorem{lemma}{Lemma}[section]

\newtheorem{remark}{Remark}[section]
\usepackage[utf8]{inputenc}
\allowdisplaybreaks

\usepackage{graphicx}
\DeclareCaptionSubType*[Alph]{figure}
\makeindex                           
\begin{document}

\numberwithin{equation}{section}
\newtheorem{thm}{Theorem}[section]
 \newtheorem{cor}[thm]{Corollary}
 \newtheorem{lem}[thm]{Lemma}
 \newtheorem{prop}[thm]{Proposition}
 \theoremstyle{definition}
 \newtheorem{defn}[thm]{Definition}
 \theoremstyle{remark}
 \newtheorem{rem}[thm]{Remark}
 \newtheorem{cla}[thm]{\bf Claim}
 \newtheorem{example}[theorem]{Example}

\def\Re{\mathbb{R}}
\def\Ze{\mathbb{Z}}
\renewcommand{\arraystretch}{1.4}
\def\minmod{\mbox{minmod}}
\def\f#1#2{\frac {#1}{#2}}
\def\R{{\bf R}}
\def\n{\noindent}
\def\dw{\downarrow}
\def\uw{\uparrow}
\def\hang{\hangindent\parindent}
\def\textindent#1{\indent\llap{#1\enspace}\ignorespaces}
\def\re{\par\hang\textindent}
\def\d{\displaystyle}
\def\rp{\rightharpoonup}
\def\rup#1{\stackrel{\rightharpoonup}#1}
\def\lup#1{\stackrel{\leftharpoonup}#1}
\def\rls#1{\stackrel{\rightleftharpoons}#1}
\def\up#1{\stackrel#1{=}}
\def\btp{\bigtriangleup}
\def\lf{\lfloor}
\def\rf{\rfloor}
\def\b#1{\overline#1}
\def\bga{\begin{array}}
\def\eda{\end{array}}
\def\pr{\protect}
\def\ka{\kappa}
\def\ze{\zeta}
\def\de{\delta}
\def\ded{\delta _0}
\def\De{\Delta}
\def\pt{\partial}
\def\bs{\backslash}
\def\orw{\overrightarrow}
\def\ev{\equiv}
\def\ol{\overline}
\def\Gm{\Gamma}
\def\gm{\gamma}
\def\ga{\gamma}
\def\th{\theta}
\def\nb{\nabla}
\def\ep{\epsilon}
\def\la{\lambda}
\def\La{\Lambda}
\def\sg{\sigma}
\def\Sg{\Sigma}
\def\ba{\mathbf{a}}
\def\bb{\mathbf{b}}
\def\bc{\mathbf{c}}
\def\bd{\mathbf{d}}
\def\be{\mathbf{e}}
\def\bbf{\mathbf{f}}
\def\bg{\mathbf{g}}
\def\bh{\mathbf{h}}
\def\bi{\mathbf{i}}
\def\bj{\mathbf{j}}
\def\bk{\mathbf{k}}
\def\bl{\mathbf{l}}
\def\bbm{\mathbf{m}}
\def\bn{\mathbf{n}}
\def\bo{\mathbf{o}}
\def\bp{\mathbf{p}}
\def\bq{\mathbf{q}}
\def\br{\mathbf{r}}
\def\bs{\mathbf{s}}
\def\bt{\mathbf{t}}
\def\bu{\mathbf{u}}
\def\bv{\mathbf{v}}
\def\bw{\mathbf{w}}
\def\bx{\mathbf{x}}
\def\by{\mathbf{y}}
\def\bz{\mathbf{z}}

\def\bA{\mathbf{A}}
\def\bB{\mathbf{B}}
\def\bC{\mathbf{C}}
\def\bD{\mathbf{D}}
\def\bE{\mathbf{E}}
\def\bF{\mathbf{F}}
\def\bG{\mathbf{G}}
\def\bH{\mathbf{H}}
\def\bI{\mathbf{I}}
\def\bJ{\mathbf{J}}
\def\bK{\mathbf{K}}
\def\bL{\mathbf{L}}
\def\bM{\mathbf{M}}
\def\bN{\mathbf{N}}
\def\bO{\mathbf{O}}
\def\bP{\mathbf{P}}
\def\bQ{\mathbf{Q}}
\def\bR{\mathbf{R}}
\def\bS{\mathbf{S}}
\def\bT{\mathbf{T}}
\def\bU{\mathbf{U}}
\def\bV{\mathbf{V}}
\def\bW{\mathbf{W}}
\def\bX{\mathbf{X}}
\def\bY{\mathbf{Y}}
\def\bZ{\mathbf{Z}}

\def\wtd{\widetilde}
\def\al{\alpha}
\def\td{\tilde}
\def\rw{\rightarrow}
\def\iy{\infty}
\def\Om{\Omega}
\def\om{\omega}
\def\sb{\subset}
\def\sp{\supset}
\def\sq{\sqrt}
\def\sbq{\subseteq}
\def\spq{\supseteq}
\def\d{\displaystyle}
\def\dfr#1#2{\displaystyle{\frac{#1}{#2}}}
\def\u{\underline}
\def\li#1#2{\displaystyle{\lim_{#1\rightarrow#2}}}
\renewcommand{\figurename}{Fig.}




\title{A generalized Riemann problem solver for a hyperbolic model of two-layer thin film flow}

\author{Rahul Barthwal, Christian Rohde, Yue Wang}
\author{Rahul Barthwal\thanks{Institute of Applied Analysis and Numerical Simulation, University of Stuttgart, 70569, Stuttgart, Germany; Email: rahul.barthwal@mathematik.uni-stuttgart.de}, Christian Rohde\thanks{Institute of Applied Analysis and Numerical Simulation, University of Stuttgart, 70569, Stuttgart, Germany; Email: christian.rohde@mathematik.uni-stuttgart.de}, and Yue Wang\thanks{National Key Laboratory of Computational Physics, Institute of Applied Physics and Computational Mathematics, 100094, Beijing, P. R. China; Email: wang\_yue@iapcm.ac.cn}}
\maketitle




\begin{abstract}
In this paper, a second-order generalized Riemann problem (GRP) solver is developed for a two-layer thin film model. Extending the first-order Godunov approach, the solver is used to construct a temporal-spatial coupled second-order GRP-based finite-volume method. Numerical experiments including comparisons to MUSCL finite-volume schemes with Runge-Kutta time stepping confirm the accuracy, efficiency and robustness of the 
higher-order ansatz.

The construction of GRP methods requires to compute temporal derivatives of intermediate states in the entropy solution of the generalized Riemann problem. These derivatives are obtained from the Rankine-Hugoniot conditions as well as a
characteristic decomposition using Riemann invariants. Notably, the latter can be computed explicitly for the two-layer thin film model, which renders this system to be very suitable for the GRP approach. Moreover, it becomes possible to determine the derivatives in an explicit, computationally cheap way. 

\end{abstract}

\pagestyle{myheadings} \thispagestyle{plain} \markboth{
   GRP for thin film flow} {  }
\section{Introduction}\label{sec: 1}
The generalized Riemann problem (GRP) method was originally proposed and developed for compressible fluid dynamics \cite{Ben-Artzi-1984}. By now, it has been utilized for many examples of hyperbolic dynamics, including the motion of elastic strings \cite{ta1992generalized},  shallow water flows \cite{li2006generalized}, reactive flows\cite{ben1989generalized}, multi-fluid and multi-phase flows\cite{Qi2014,wang2023stiffened,lei2021staggered,zhang2024generalized}, special relativistic hydrodynamics \cite{yang2011direct,yang2012direct, wu2016direct}, radiative hydrodynamics \cite{kuang2019second},  blood flow models in arteries \cite{sheng2021direct} and moment closure equations \cite{wang2025second}. 

The second-order GRP method relies on the computation of so-called instantaneous time derivatives of the conservative variables.  These appear as states in the intermediate regions between two elementary waves in the wave pattern of the second order Riemann problem, i.e., the initial value problem for the hyperbolic balance law with piecewise affine initial data.  To find the derivatives, the knowledge of Riemann invariants is crucial: the hyperbolic system can then be converted into a diagonal form
(see Chapter 7 in \cite{dafermos2005hyperbolic} for more details). In fact, Ben-Artzi\&Li \cite{ben2007hyperbolic} proved that GRP solvers can be always developed for such systems.

Recently, Barthwal\&Rohde \cite{barthwal2025hyperbolic} suggested a $(4\times 4)$-hyperbolic system of conservation laws which governs the first-order dynamics of a two-layer thin film with (anti-) surfactants. It is given by
\begin{align}\label{eq: Main_system}
\begin{cases}
    \dfrac{\partial f}{\partial t}+\dfrac{1}{2}\dfrac{\partial}{\partial x}\left(f^2b\right)&=0,\vspace{0.2 cm} \\
    \dfrac{\partial b}{\partial t}+\dfrac{1}{2}\dfrac{\partial}{\partial x}\left(fb^2\right)&=0,\vspace{0.2 cm} \\
     \dfrac{\partial g}{\partial t}+\dfrac{\partial}{\partial x}\left(\dfrac{g^2q}{2}+fgb\right)&=0,\vspace{0.2 cm} \\
  \dfrac{\partial q}{\partial t}+\dfrac{\partial}{\partial x}\left(\dfrac{gq^2}{2}+fbq\right)&=0,
  \end{cases}
\end{align}
where $f$ and $g$ denote the positive film thicknesses in the two layers.  The quantities 
$b$ 
 and
$q$ 
denote the spatial derivatives of the concentrations of the solute in each layer. The first-order system \eqref{eq: Main_system} can be written in the compact conservative form 
\begin{align}
&\mathbf{U}_t+(\mathbf{F}{(\mathbf{U}))}_x=0,\label{conservation-laws}
\end{align}
where $\mathbf{U}=(f,b,g,q)^{\top}\in \mathcal{U}$ denotes the vector of conservative variables and 
\begin{align}\label{flux}
\mathbf{F}(\mathbf{U})=\left(\frac 12 f^2b,\frac 12fb^2,\frac 12 g^2q+fbg,\frac 12 gq^2+fbq\right)^{\top}
\end{align}
is the flux vector. 

The system \eqref{eq: Main_system} can be decoupled along its Riemann invariants (see Section \ref{sec: 2} below), and the set of Riemann invariants forms a coordinate system of \eqref{eq: Main_system} \cite{barthwal2025hyperbolic}. Based on such a decomposition along Riemann invariants, an entropy/entropy-flux pair with a strictly convex entropy for the system has been found in \cite{barthwal2025hyperbolic}. In what follows, we consider entropy solutions of \eqref{eq: Main_system} in the state space
\begin{align}\label{statespace}
    \mathcal{U}=\{(f, b, g, q)\in \mathbb{R}^4: f, g, q>0, b<0, fb+gq\geq 0\}.
\end{align}
This state space allows us to consider the negative concentration gradient in the first phase. However, the concentration gradient in the second phase is considered to be positive for technical reasons. Physically, the state space refers to the situation where the interfacial stress and the surface tension stress are competing with each other. The Marangoni stress in the first layer drives the flow towards the left, while the Marangoni stress in the second layer drives the flow towards the right. Therefore, it is interesting to understand the mutual effects. Other possible state spaces, where the system is hyperbolic, are discussed in Appendix \ref{appendix-B}.

The modelling, analysis, and numerics of thin film flows have gained a lot of attention from researchers working in the field of hyperbolic balance laws. Notably, hyperbolic models for thin film flow down an inclined plane 
or under the influence of gravity or surfactants have been developed and analyzed,  see e.g., \cite{dhaouadi2022hyperbolic, barthwal2022two, bertozzi1999undercompressive, barthwal2023construction, pandey2025construction, levy2006motion, cook2008shock} and references cited therein. It is vital to understand the dynamics of such flows as they find important applications in technical and environmental fields like coating technologies, thin film solar cell technologies, surfactant replacement therapies, etc.   To the best of the authors' knowledge,  there is almost no work on taylored high-order schemes for  such models. Not only this, but most of the recently developed second-order or higher-order GRP solvers are restricted to the case of hyperbolic systems where the mathematical structure closely resembles that of the compressible Euler equations. 
We exploit the special structure of \eqref{eq: Main_system} and develop a novel second-order GRP solver for the system \eqref{eq: Main_system}. Precisely, we  leverage the fact that the system \eqref{eq: Main_system} can be decoupled along its Riemann invariants. In particular, an invertible system of linear equations is formulated to obtain the instantaneous time derivatives of the conservative variables by using the Riemann invariants and Rankine-Hugoniot conditions. 

The discussion in this article is not just of theoretical significance in explicitly computing instantaneous time derivatives of the conservative variables, but also serves as the basis of spatio-temporal coupled
high-order numerical schemes \cite{Ben-Artzi-1984, Ben-Artzi-2003, ben2006direct}. Notably, GRP methods can improve accuracy and reduce computational costs compared to other approaches with the same numerical errors \cite{wang2023stiffened}. In recent years, numerous studies have been successfully conducted in this direction. The adaptive direct Eulerian GRP method was developed in \cite{han2010adaptive} by coupling the GRP method with the moving mesh method \cite{tang2003adaptive}, leading to improved resolution and accuracy. In addition, the adaptive GRP method was extended to unstructured triangular meshes \cite{li2013adaptive} and utilized to simulate 2D complex wave configurations formulated with the 2D Riemann problems of the Euler equations \cite{han2011accuracy}. The third-order GRP methods for the Euler equations \cite{wu2014third2} and general hyperbolic balance laws \cite{goetz2018family,  goetz2016novel, qian2023high} are some of the examples of the numerous efforts on high-order GRP methods. Furthermore, a two-stage fourth-order time-accurate technique was presented using the second order GRP solver as the building block. It was firstly combined with the finite volume framework for hyperbolic conservation laws \cite{li2016two}. Recently, arbitrarily high-order DG schemes based on the GRP solver were developed in \cite{wang2015arbitrary}, where the reconstruction steps were halved compared with the existing high-order Runge-Kutta DG (RKDG) schemes. When compared to the same order multi-stage strong-stability-preserving (SSP) RKDG technique, the computational cost of the two-stage fourth-order DG method based on the GRP solver can be significantly decreased \cite{cheng2019two}. This indicates the versatility of the GRP method.
\medskip

We conclude the introductory part with the  basic setup for the proposed second-order finite-volume solver. It shows how the solution of the GRP that we consider in the subsequent sections enters the algorithm. 

For the sake of simplicity, we divide the space domain  $\mathbb R$ into a uniform grid 
$\{x_{j+\frac12}, j\in\mathbb Z\}$ with  
$\Delta x:=x_{j+\frac12}-x_{j-\frac12}$.
Define the cell $I_j:=[x_{j-\frac{1}{2}},x_{j+\frac{1}{2}}]$ 
with mid point $x_j = (x_{j+\frac12}+ x_{j-\frac12})/2$. 
Let for  cell averages  $\mathbf{U}_j^k\in \mathcal{U}$ the piecewise constant approximation
\[
\bar\bU_{\Delta x}(x, t^k) = \mathbf{U}_j^k, \qquad
x\in I_j
\]
be given at some time $t^k\ge 0$ for $k\in \mathbb N$ and $t^0 =0$.

Then we summarize the second-order finite-volume method in the following steps. 
\begin{enumerate}
    \item[(i)] {\em (Piecewise linear reconstruction of data).} Let $j\in \mathbb{Z}$. From the piecewise constant approximate  solution, $\bar\bU_{\Delta x}(x, t^k)$,  we  define the  piecewise first-order polynomial
  \begin{equation}\label{piecewise linear initial data}
\mathbf{{U}}_{\Delta x}(x,t^k)=\bar\bU_{\Delta x}(x, t^k)+{\boldsymbol{\sigma}}_j^k(x-x_j), \quad x\in I_j.
\end{equation}
Here, 
the parameter $\boldsymbol{\sigma}_j^k$ is computed by a slope limiter.

\item[(ii)]{\em (Computation of the generalized Riemann problem).} 
 Let $j\in \mathbb{Z}$. Use Section \ref{sec: 2.4} to compute the vector 
$\mathbf{U}_{j+1/2}^{k} :=R^A\left(0;\mathbf{U}_{j+\frac{1}{2},L}^{k},\mathbf{U}_{j+\frac{1}{2},R}^{k}\right)$, which denotes the self-similar entropy solution of the Riemann problem  for \eqref{eq: Main_system} centered at $(x_{j+\frac{1}{2}},t^k)$ with the initial datum

\begin{equation}\label{local RP1}
\mathbf{U}(x,0)=\begin{cases}
                    \mathbf{U}_{j+\frac{1}{2},L}^{k},
                     \quad x<x_{j+\frac{1}{2}}, \\
                    \mathbf{U}_{j+\frac{1}{2},R}^{k},
                     \quad x>x_{j+\frac{1}{2}}.
                  \end{cases}
\end{equation}
In \eqref{local RP1}, $\mathbf{U}_{j+\frac{1}{2},L}^{k}=\mathbf{U}_{j}^{k}+\frac{\Delta x}{2}\boldsymbol{\sigma}_j^k$ and $\mathbf{U}_{j+\frac{1}{2},R}^{k}=\mathbf{U}_{j+1}^{k}-\frac{\Delta x}{2}\boldsymbol{\sigma}_{j+1}^k$ are the left and right trace values of $\mathbf{U}_{\Delta x}(\cdot,t^k)$ at $x_{j+\frac{1}{2}}$.
Furthermore, compute the instantaneous time derivatives $\left(\frac{\partial \mathbf{U}}{\partial t}\right)_{j+1/2}^{k}$ of the entropy solution $\mathbf{U}$ at the cell interface $x=x_{j+1/2}$ with the initial datum
\begin{equation}\label{local GRP1}
\mathbf{U}(x,0)=\begin{cases}
                    \mathbf{U}_{j}^{k}+{\boldsymbol{\sigma}}_j^k(x-x_j),
                     \qquad\qquad x<x_{j+\frac{1}{2}}, \\
                    \mathbf{U}_{j+1}^{k} +{\boldsymbol{\sigma}}_{j+1}^k(x-x_{j+1}),
                     ~\quad x>x_{j+\frac{1}{2}},
                  \end{cases}
\end{equation}
according to Theorems \ref{non-sonic-case}, \ref{acoustic} and \ref{non-sonic-case-wave-2} by the GRP solver. 


  \item[(iii)]  {\em (Finite-volume step: Advancing cell-averages).} 
Let $\lambda_i, ~i\in \{1, \ldots, 4\}$ be the characteristic speeds of the system \eqref{eq: Main_system} and $\Delta t_k$ satisfies for $\rm{CFL}>0$, the  time 
 step condition   \[
\max_{i\in \{1,\ldots,4\},\, j\in\mathbb{Z}  }|\lambda_{i}(\mathbf{U}_j^k)|\Delta t_k\leq 
 {\rm{CFL}}\,\Delta x.\] We compute $t^{k+1} = t^k + \Delta t_k$
and update the cell-average at the next time step according to the formula
  \begin{equation}\label{1d_scheme}
\mathbf{U}_j^{k+1}=\mathbf{U}_j^{k}-\frac{\Delta t_k}{\Delta x}\left(\mathbf{F}_{j+\frac{1}{2}}^{k+\frac{1}{2}}-\mathbf{F}_{j-\frac{1}{2}}^{k+\frac{1}{2}}\right).
\end{equation}
Here $\mathbf{F}_{j+\frac{1}{2}}^{k+\frac{1}{2}}$ is given by 
\begin{align}
    &\mathbf{F}_{j+\frac{1}{2}}^{k+\frac{1}{2}}=
\mathbf{F}\left(\mathbf{U}_{j+1/2}^{k+1/2}\right),\label{flux_computation}\\
&\mathbf{U}_{j+1/2}^{k+1/2}=\mathbf{U}_{j+1/2}^{k}+\frac{\Delta t_k}{2}\left(\frac{\partial \mathbf{U}}{\partial t}\right)_{j+1/2}^{k}.\label{solution_midpoint}%
\end{align}
 

\item[(iv)]{\em (Slope update).} Evaluate and update for $\theta \in [0, 2)$, the slope by using the monotonicity slope limiters as follows \cite{hesthaven2017numerical, leveque1992numerical},
\begin{align}\label{slope_update}
\boldsymbol{\sigma}_j^{k+1}=\text{minmod}\left(\theta \dfrac{\mathbf{U}_j^{k+1}-\mathbf{U}_{j-1}^{k+1}}{\Delta x}, \dfrac{\mathbf{U}_{j+1/2}^{k+1, -}-\mathbf{U}_{j-1/2}^{k+1, -}}{\Delta x}, \theta \dfrac{\mathbf{U}_{j+1}^{k+1}-\mathbf{U}_{j}^{k+1}}{\Delta x}\right),
\end{align}
where
\begin{align}\label{updated_slope_value}
    \mathbf{U}_{j+1/2}^{k+1, -}&=\mathbf{U}_{j+1/2}^k+\Delta t_k \left(\dfrac{\partial \mathbf{U}}{\partial t}\right)_{j+1/2}^k.
\end{align}
\end{enumerate}

The formula in \eqref{flux_computation} is supposed to approximate the time-integrated flux
\[.
\frac{1}{\Delta t}\int_{t^k}^{t^{k+1}}\bF\left(\bU(x_{j+\frac 12},t)\right)\,dt
\]
up to second order. Therefore, we can expect to obtain a second-order in space and time accurate scheme. 

The remainder of the article is organized as follows. Section \ref{sec: 2} is devoted to provide preliminaries for the system and an exact Riemann solver for the system \eqref{eq: Main_system}. Section \ref{sec: 3} discusses the solution of GRP for the first of two possible wave configurations, 
while Section \ref{sec: 4} is devoted to the other essentially different wave configuration.
Numerical experiments
are presented in Section \ref{sec: 5} to demonstrate the accuracy and performance of the proposed GRP method. In particular, we construct new exact travelling waves for the system \eqref{eq: Main_system} and use them to validate the numerical results obtained via the GRP method. Concluding remarks and future scopes are provided in Section \ref{sec: 6}. In Appendix \ref{appendix-B}, we discuss other possible state spaces, where the system \eqref{eq: Main_system} is hyperbolic.
\section{Characteristic analysis and the Riemann problem for \eqref{eq: Main_system}}\label{sec: 2}
The Riemann problem plays a fundamental role in designing Godunov-type numerical algorithms. It helps us to understand the interaction of elementary waves, which is also the basis of the GRP solver. 
In this section, we summarize the results from \cite{barthwal2025hyperbolic}, which includes the characteristic analysis of the hyperbolic system \eqref{eq: Main_system} and the related exact Riemann solution.
\subsection{Characteristic analysis for the system \eqref{eq: Main_system}}\label{subsec: 3.2}
For smooth solutions, we can convert the system $\eqref{eq: Main_system}$ into its primitive form as
\begin{align}\label{eq: 3.6}
&\mathbf{U}_t+\mathbf{DF}(\mathbf{U})\mathbf{U}_x=0,\\
&\mathbf{DF}(\mathbf{U})=\left(
\begin{array}{cccc}
fb & \frac 12 f^2 & 0 & 0\\
\frac 12 b^2 & fb & 0 & 0\\
gb & fg & fb+gq & \frac 12 g^2\\
bq & fq & \frac 12 q^2 & fb+qg \\
\end{array}
\right). \nonumber
\end{align}
The eigenvalues $\lambda_i= \lambda_i(\mathbf{U})$ of the block-structutred Jacobian $\mathbf{DF}(\mathbf{U})$ can be easily computed. Considering the state space \eqref{statespace}, the eigenvalues are real and given by
\begin{align}
\lambda_1(\mathbf{U})=\dfrac{3fb}{2}<~\lambda_2(\mathbf{U})=\dfrac{fb}{2}\leq ~\lambda_3(\mathbf{U})=fb+\dfrac{gq}{2}<~\lambda_4(\mathbf{U})=fb+\dfrac{3gq}{2}.
\end{align}
Moreover, the right eigenvectors $\mathbf{r_i}= \mathbf{r_i}(\mathbf{U})$ of $\mathbf{DF}(\mathbf{U})$ are
\begin{align}
&\mathbf{r_1}(\mathbf{U})=\left(\frac{fb-3gq}{4qb},\frac{fb-3gq}{4qf},\frac{g}{q},1\right)^{\top},\ \mathbf{r_2}(\mathbf{U})=\left(-\frac{f}{b},1,0,0\right)^{\top},\notag\\
&\mathbf{r_3}(\mathbf{U})=\left(0,0,-\frac{g}{q},1\right)^{\top},\ \mathbf{r_4}(\mathbf{U})=\left(0,0,\frac{g}{q},1\right)^{\top}.\label{righteigenvectors}
\end{align}
The eigenvectors are linearly independent and thus form a complete basis of $\mathbb{R}^4$ for $\mathbf{U}\in \mathcal{U}$, which implies the strict hyperbolicity of the system \eqref{eq: Main_system}.

It is easy to see that $\nabla_{\mathbf{U}} \lambda_2\cdot\mathbf{r_2}=\nabla_{\mathbf{U}} \lambda_3\cdot\mathbf{r_3}=0$ and $\nabla_{\mathbf{U}} \lambda_1\cdot \mathbf{r_1}=3fb\neq 0$ and $\nabla_{\mathbf{U}} \lambda_4\cdot\mathbf{r_4}=3gq\neq 0$ holds, which implies that the second and third characteristic fields are linearly degenerate while the first and fourth characteristic fields are genuinely nonlinear in $\mathcal{U}$. Note that $\nabla_{\mathbf{U}}$ denotes the gradient of any vector field with respect to the conservative vector $\mathbf{U}=(f, b, g, q)^T$.
\subsection{Riemann invariants}
In order to resolve the rarefaction waves for system \eqref{eq: Main_system}, it is often convenient to introduce a set of Riemann invariants, which can form a coordinate system. Indeed, the system \eqref{eq: Main_system} is equipped with a full set of three $k$-Riemann invariants $\Gamma_1^k,\Gamma_2^k,\Gamma_3^k$ corresponding to the $k\rm{th}$-characteristic field, $k=1,\ldots, 4$. They are given by 
\begin{align}
\begin{cases}
\Gamma_1^1 ={f}/{b},~\Gamma_2^1={g}/{q},~\Gamma_3^1={(fb+gq)}/{(gq)^{1/4}},\\
\Gamma_1^2=g,~\quad\Gamma_2^2=q,~\quad\Gamma_3^2=fb,\\
\Gamma_1^3=f,~\quad\Gamma_2^3=b,~\quad\Gamma_3^3=gq,\\
   \Gamma_1^4=f,~\quad\Gamma_2^4=b,~\quad\Gamma_3^4={g}/{q}.
\end{cases}   \label{Rinvariants} 
\end{align}
From \eqref{Rinvariants} we select the four Riemann invariants 
\begin{equation}\label{selectedinvariants}
\xi=w_1=f/b, \quad     u= w_2 = fb,  \quad \tau = w_3 = g/q, \quad 
\eta = w_4=(fb+gq)/(gq)^{1/4}.
\end{equation}
These form a coordinate system for \eqref{eq: Main_system}, i.e., the mapping 
\[
\mathbf{U}  = (f,b,g,q)^\top \mapsto \mathbf{W}= (u,\xi, \tau,\eta)^\top
\]
is one-to-one in $\mathcal U$ and the system \eqref{eq: Main_system} can be transformed into the diagonal form (see \cite{barthwal2025hyperbolic} for more details)
\begin{equation}\label{Rtransformed}
\begin{array}{rclrcl}
u_t+\dfrac{3 u}{2} u_x&=&0, \vspace{0.2 cm}\\
   \xi_t+\dfrac{u}{2} \xi_x&=&0,\vspace{0.2 cm}\\ 
   \tau_t+\left(u+\dfrac{v}{2}\right)\tau_x&=&0, \vspace{0.2 cm}\\
\eta_t+\left(u+\dfrac{3v}{2}\right)\eta_x&=&0,
\end{array}    
\end{equation}
where $v=gq$. Here, $v$ satisfies the following non-conservative equation
\begin{align}
    &v_t+2vu_x+(u+\frac 32v)v_x=0.\label{simeqn2}
\end{align}
The non-conservative equation for $v$ helps us to find relations along Riemann invariants, as we will see in Section \ref{sec: 3} and Section \ref{sec: 4}.
\subsubsection{Rarefaction waves}\label{subsubsec: rarefaction}
Before we construct rarefaction wave curves, we
recall that the Riemann invariants  $\Gamma^k_1,\Gamma^k_2,\Gamma^k_3$ from \eqref{Rinvariants} remain constant across a corresponding $k$-rarefaction wave \cite{dafermos2005hyperbolic}.
A $k$-rarefaction wave can only exist for a genuinely nonlinear field. Therefore we consider the cases $k\in \{1,4\}$, see \cite{dafermos2005hyperbolic}.

We start with $k=1$. The slope inside a 1-rarefaction  wave is given by
\begin{align*}
    \dfrac{dx}{dt}=\dfrac{x}{t}=\lambda_1=\dfrac{3fb}{2}.
\end{align*}
The characteristic speed increases across the rarefaction wave, which means that $fb\geq f_lb_l$ must hold for any left state $\mathbf{U}_l=(f_l, b_l, g_l, q_l)^{\top}\in \mathcal{U}$ across the 1-rarefaction wave. 
Thus, using the Riemann invariants $\Gamma^1_1,\Gamma^1_2,\Gamma^1_3$ from \eqref{Rinvariants}, the solution inside the 1-rarefaction wave is given for the left state $\mathbf{U}_l$ by 
\begin{align}\label{eq: 4.2R} R_1:=
    \begin{cases}
        \dfrac{dx}{dt}=\dfrac{x}{t}=\lambda_1=\dfrac{3fb}{2},\vspace{0.2 cm}\\
        \dfrac{f}{b}=\dfrac{f_l}{b_l},~\dfrac{fb+gq}{(gq)^{1/4}}=\dfrac{f_lb_l+g_lq_l}{(g_lq_l)^{1/4}}, ~\dfrac{g}{q}=\dfrac{g_l}{q_l},\vspace{0.2 cm}\\
        fb\geq f_lb_l.
    \end{cases}
\end{align}
$R_1$ is a curve in  the $(f, b, g, q)$-space emanating from $\mathbf{U}_l$.

Similarly, the 4-rarefaction wave curve starting from a left state $\mathbf{U}_l$ is given  by
\begin{align} \label{R4} R_4:=
    \begin{cases}
        \dfrac{dx}{dt}=\dfrac{x}{t}=\lambda_4=fb+\dfrac{3gq}{2},\vspace{0.2 cm}\\
        \dfrac{g}{q}=\dfrac{g_l}{q_l},~f=f_l, b=b_l,\vspace{0.2 cm}\\
        g_lq_l\leq gq.
    \end{cases}
\end{align}

\subsubsection{Discontinuous waves: shock and contact waves}\label{sec: 2.3}
For speed $\sigma \in \mathbb{R}$,  a discontinuous wave 
\begin{align} \label{wave}
    \mathbf{U} (x, t)=\begin{cases}
        \mathbf{U}_l=(f_l, b_l, g_l, q_l )^\top\in {\mathcal U}&:~x-\sigma t<0,\\
        \mathbf{U}_r=(f_r, b_r, g_r, q_r)^\top \in {\mathcal U}&:~x-\sigma t> 0,
    \end{cases}
\end{align}
is a distributional solution of \eqref{eq: Main_system} if the  Rankine-Hugoniot conditions 
\begin{align}\label{RH}
\sigma[\![\mathbf{U}]\!]  =[\![\mathbf{F(U)}]\!]
\end{align}
hold. Here,  $[\![\mathbf{W}]\!]=\mathbf{W}_r-\mathbf{W}_l$ denotes the jump in $\mathbf{W}$ for $\mathbf{W}_l, \mathbf{W}_r\in \mathbb{R}^4$.

For the system \eqref{eq: Main_system},  the componentwise Rankine-Hugoniot relations  are
\begin{equation}\label{eq: 4.8}
    \begin{array}{rclrcl}
    \sigma[\![f]\!]\!\!\!\!&=&\!\!\!\!\dfrac{1}{2}[\![f^2b]\!],&
    \sigma[\![b]\!]\!\!\!\!&=&\!\!\!\!\dfrac{1}{2}[\![fb^2]\!],\vspace{0.2 cm}\\
    \sigma[\![g]\!]\!\!\!\!&=&\!\!\!\!\Big[\!\!\Big[\dfrac{g^2q}{2}+fgb\Big]\!\!\Big],&
    \sigma[\![q]\!]\!\!\!\!&=& \!\!\!\!\Big[\!\!\Big[\dfrac{gq^2}{2}+fbq\Big]\!\!\Big].
    \end{array}
\end{equation}
Along the second and third linearly degenerate characteristic fields, for a given left state $\mathbf{U}_l=(f_l, b_l, g_l, q_l)^{\top}\in \mathcal{U}$, the discontinuity curves are  curves through $\mathbf{U}_l$ in the $(f, b, g, q)$-space, which are given by
\begin{align}\label{eq: 23a}
    J_2:=\begin{cases}
    \sigma_2=\dfrac{fb}{2}=\dfrac{f_lb_l}{2},\\ 
    fb=f_lb_l,~g=g_l, ~q=q_l
    \end{cases}
\end{align}
and 
\begin{align}\label{eq: 23b}
    J_3:=\begin{cases}
    \sigma_3=fb+\dfrac{gq}{2}=f_lb_l+\dfrac{g_lq_l}{2},\\ 
    f=f_l, ~b=b_l,  ~gq=g_lq_l.
    \end{cases}
\end{align}
The corresponding 2- and 3-contact waves \eqref{wave} are entropy solutions of \eqref{eq: Main_system} by definition.

For $k\in \{1,4\}$, the $k$th characteristic field is genuinely nonlinear. We are interested in left/right states that satisfy the  Rankine-Hugoniot condition \eqref{RH} and are entropy admissible. The function $\mathbf{U}$ from \eqref{wave} is then called a $k$-shock wave.  For $k=1$ we get a curve starting from the left state $\mathbf{U_l}$ given by 
\begin{align}\label{S_2def}
   S_1:=\begin{cases} \sigma_1= \dfrac{b_l(f_l^2+f_lf+f^2)}{2f_l},\\ 
   \dfrac{f}{b}=\dfrac{f_l}{b_l}, ~\dfrac{g}{q}=\dfrac{g_l}{q_l},~g=g_l+\Phi(f_l, b_l, g_l, q_l, f, b, g, q),
   \\
   fb<f_lb_l,\end{cases}
\end{align}
with $\Phi$ being the  nonlinear function
\begin{align*}
    \Phi(f_l, b_l, g_l, q_l, f, b, g, q):= \dfrac{2f_l}{b_l(f_l^2+f_lf+f^2)}\bigg(\dfrac{g^2 q}{2}-\dfrac{g_l^2q_l}{2}+fb g-f_lb_lg_l\bigg).
\end{align*}
Given the strict hyperbolicity of the system \eqref{eq: Main_system}, the entropy admissibility of elements of $S_1$ with respect to the entropy/entropy-flux pair can be checked by the validity of the Lax entropy conditions. For the 1-shock  wave with right state $\mathbf{U}$ these are
\begin{align}\label{eq: 25}
    \lambda_1(\mathbf{U})<\sigma_1<\lambda_1(\mathbf{U}_l).
\end{align}
This results in the last inequality in \eqref{S_2def}. 

With the same arguments, we identify 4-shock waves which satisfy
\begin{align}\label{eq: 24b}
   S_4:=\begin{cases} \sigma_4= fb+\dfrac{q_l(g_l^2+g_lg+g^2)}{2g_l},\\ 
   \dfrac{q}{g}=\dfrac{q_l}{g_l},~f=f_l,~b=b_l,\\
   gq<g_lq_l.
   \end{cases}
\end{align}
We observe that the 4-shock wave relations \eqref{eq: 24b} coincide with the  4-rarefaction relations in \eqref{R4} and form a straight line in the $(f, b, g, q)$-space. This implies that the $4$th characteristic field is a Temple field \cite{temple1983systems}. However, note that the complete system does not belong to the  Temple class.
\subsection{Solution of the Riemann problem}\label{sec: 2.4}
Consider the  Riemann problem for \eqref{eq: Main_system}, that is, the Cauchy problem  with initial data of the form 
\begin{align}\label{RP-data}
    \mathbf{U} (x, 0)=\begin{cases}
        \mathbf{U}_L=(f_L, b_L, g_L, q_L )^\top\in {\mathcal U}&:~x<0,\\
        \mathbf{U}_R=(f_R, b_R, g_R, q_R)^\top \in {\mathcal U}&:~x>0.
    \end{cases}
\end{align}
Due to the fact that the second and third characteristic fields allow for contact discontinuities only, there are four wave configurations possible for the solution to the Riemann problem based on different choices of initial data. We would like to mention that the construction of the Riemann solver for state space $\mathcal{U}$ is similar to that constructed in the work of Barthwal and Rohde \cite{barthwal2025hyperbolic} using a slightly different state space. The only difference is that the first two waves change their role. For this reason, we will skip most of the details of solving the exact Riemann problem here and refer the interested reader to \cite{barthwal2025hyperbolic}.
Instead, we provide a flowchart of the exact Riemann solver in Figure \ref{fig:flowchart}.
\begin{figure}
    \centering
    \includegraphics[scale=0.5]{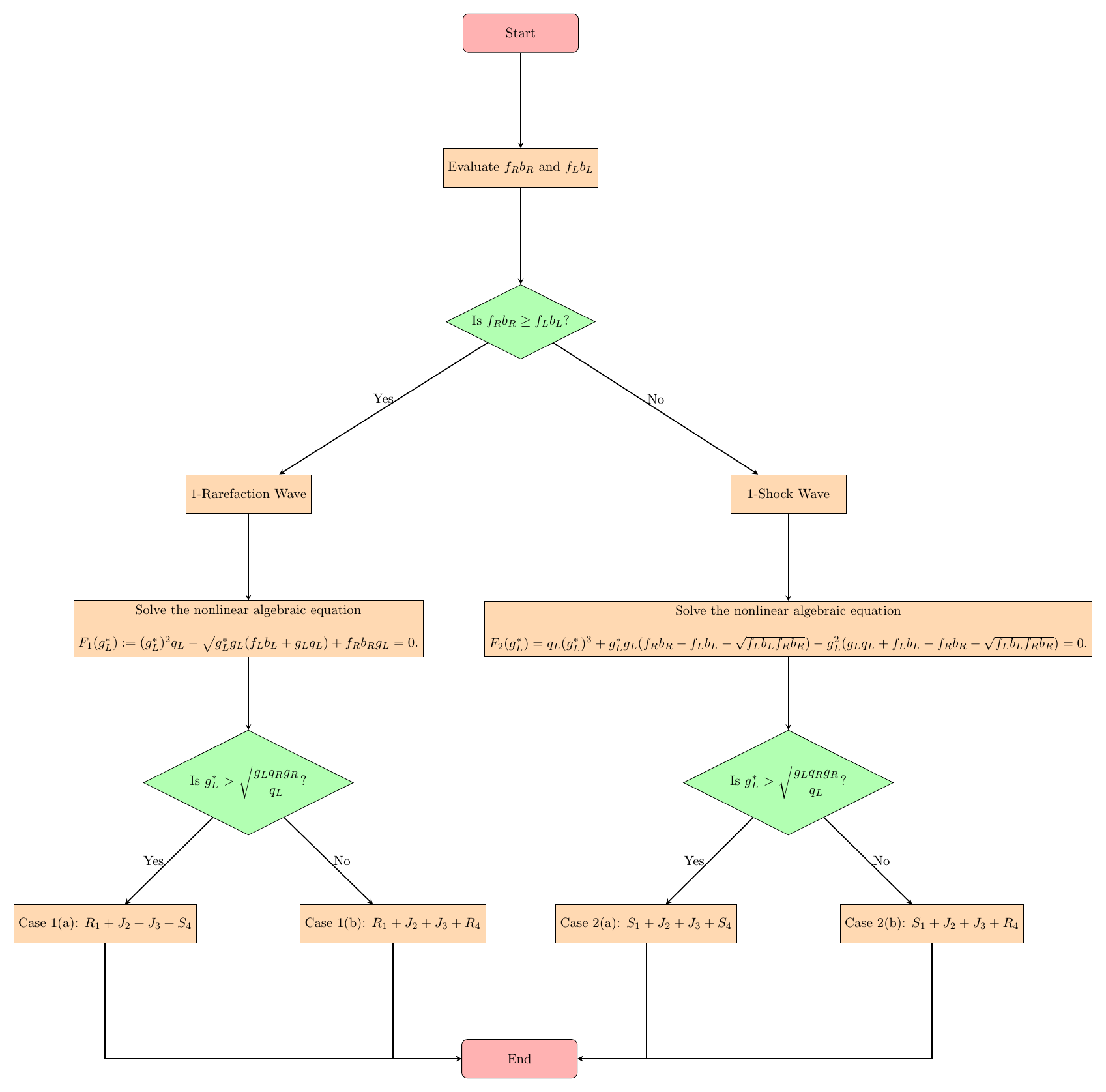}
    \caption{Flowchart of the Riemann solver for the system \eqref{eq: Main_system} that leads to four different wave patterns in Case 1(a), 1(b), 2(a) and 2(b). Here $S_i, R_i$ for $i\in \{1, 4\}$ denotes the $i$-shock and $i$-rarefaction wave, respectively.  The symbol $J_i$ for $i\in \{2, 3\}$ stands for an $i$-contact wave.}
    \label{fig:flowchart}
\end{figure}
\begin{figure}
\centering
  \begin{subfigure}[b]{0.45\linewidth}
  \includegraphics[width=3.19 in]{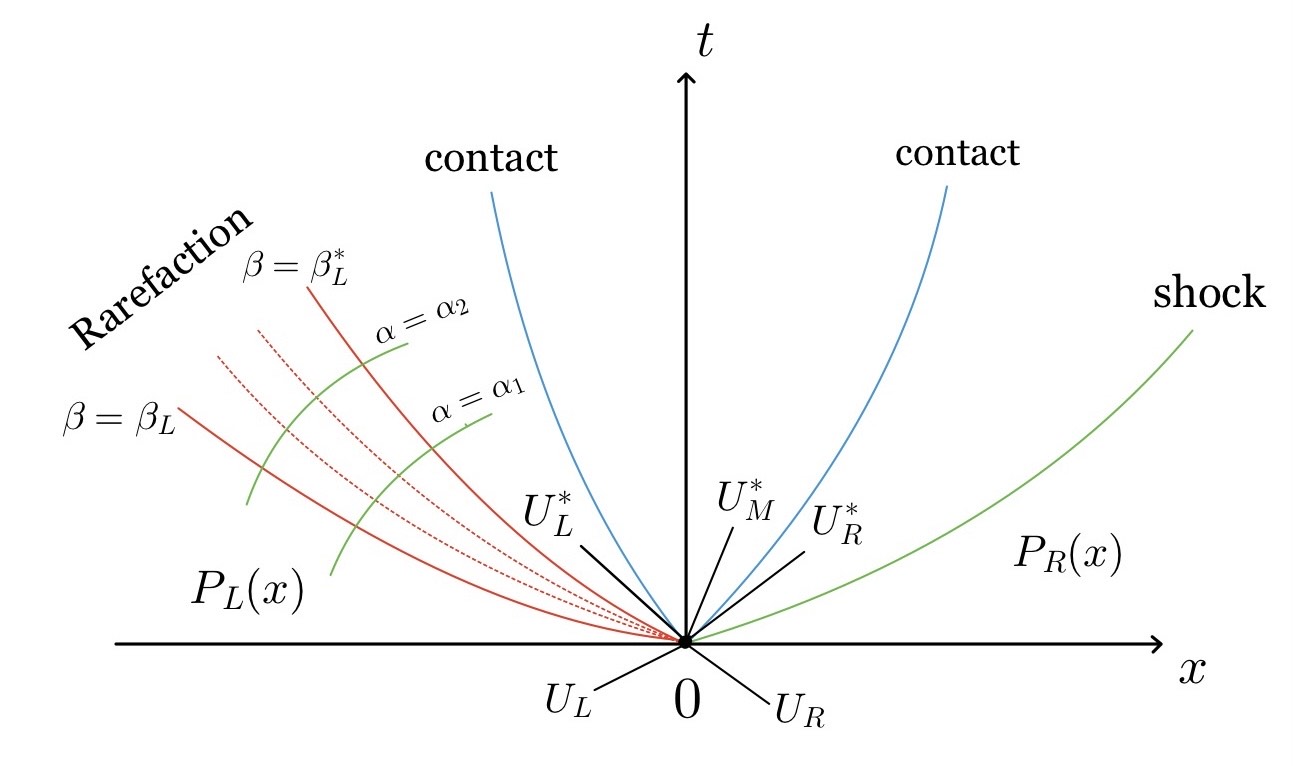}
  \caption{Typical wave configuration for the generalized Riemann problem of  Case 1(a). $R_1+J_2+J_3+S_4$.}\label{GRP-pic}
 \end{subfigure} \hspace*{0.6cm}%
 \begin{subfigure}[b]{0.45\linewidth}
     \includegraphics[width=3.1 in]{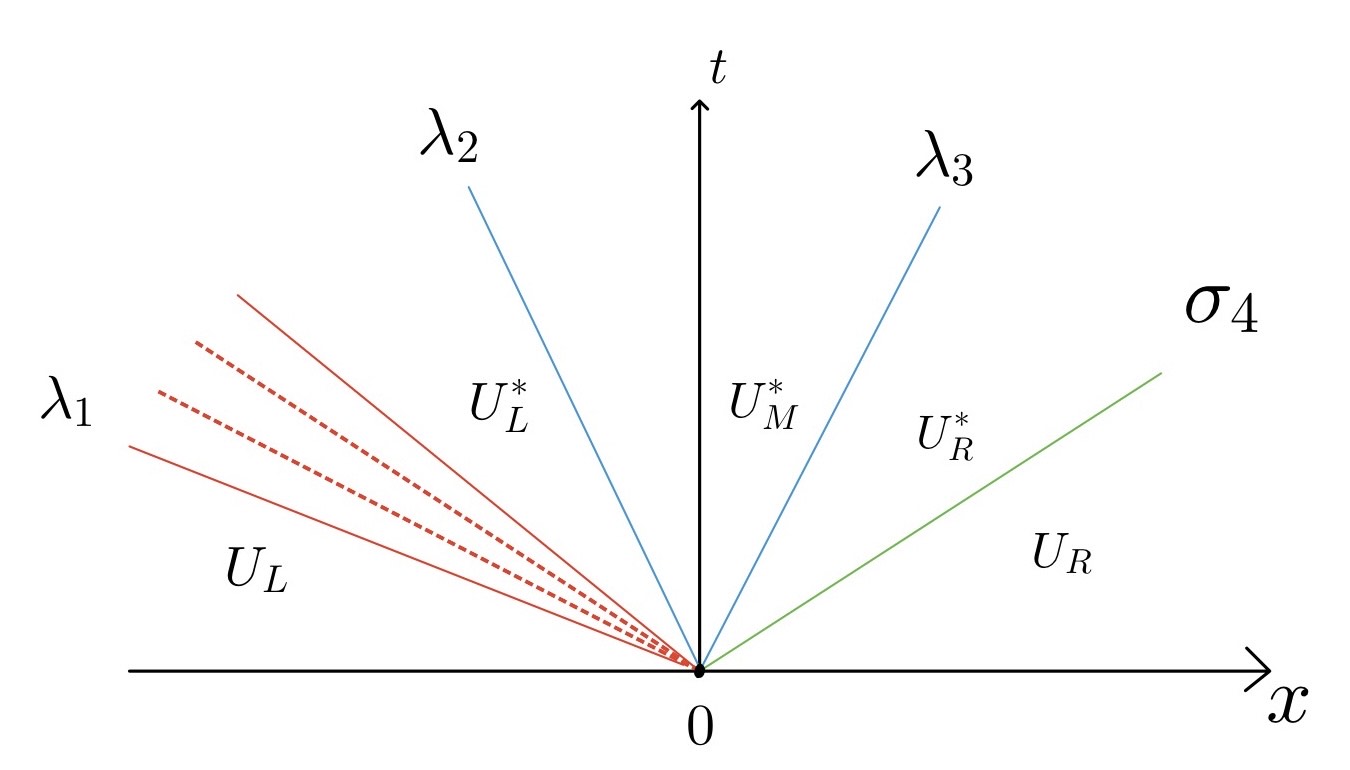}
  \caption{Typical wave configuration for the associated Riemann problem at the singularity for 
  Case 1(a). $R_1+J_2+J_3+S_4$.}\label{RP-pic}
 \end{subfigure}
\end{figure}
\subsection{The generalized Riemann problem for the system \eqref{eq: Main_system}}
For the generalized Riemann problem (GRP) for \eqref{conservation-laws}-\eqref{flux}, we consider the possibly discontinuous initial data
\begin{align}
\bU(x,0)=\left\{
\begin{array}{ll}
\bP_L(x), & x<0,\\
\bP_R(x), & x>0,
\end{array}
\right.
\end{align}
with $\bP_{L/R}$ being smooth functions, that obey together with $\bP_{L/R}'$, traces in $x=0$. The GRP solver serves to compute so-called instantaneous values in \eqref{solution_midpoint}
\begin{equation}
\bU_* = \lim_{t\rw 0^+} \bU(0,t), \ \ \ \left(\dfr{\pt \bU}{\pt t}\right)_* = \lim_{t\rw 0^+} \dfr{\pt \bU}{\pt t}(0,t).
\label{GRP-value}
\end{equation}
Define $\bU_L := \lim_{x\rw0^-}\bP_L(x)$, $\bU_R:=\lim_{x\rw 0^+}\bP_R(x)$, and $\bU_L':=\lim_{x\rw0^-}\bP'_L(x)$, $\bU'_R:=\lim_{x\rw 0^+}\bP'_R(x)$.
Previous studies \cite{Ben-Artzi-1984,Ben-Artzi-2003,ben2006direct} have shown that the instantaneous values in \eqref{GRP-value} depend only on the trace values $(\bU_L, \bU_R)$ and $(\bU_L', \bU'_R)$.  Therefore, we might as well  discuss the ``linear GRP" for  \eqref{conservation-laws}-\eqref{flux} directly  with the piecewise affine ansatz
\begin{align}
\bU(x,0)=\left\{
\begin{array}{ll}
\bU_L +\bU_L'x, \ \ \ &x<0,\\
\bU_R+\bU_R'x, \ \ \ &x>0,
\end{array}
\right.
\label{data-GRP}
\end{align}
and compute the corresponding values in \eqref{GRP-value}. 

As justified in \cite{Ben-Artzi-2003}, the solution of a GRP is based on its connection to the associated Riemann problem,
i.e., \eqref{eq: Main_system} subject to the piecewise constant initial data \eqref{RP-data}.
The self-similar entropy solution of the associated Riemann problem \eqref{RP-data} is denoted as $\bU^A(x/t) =R^A(x/t;\bU_L,\bU_R)$, and the entropy solution of the GRP \eqref{eq: Main_system}, \eqref{data-GRP} is denoted by $\bU(x,t)$. These solutions are connected by the relation
\begin{equation}
\lim_{t\rw 0^+}\bU(\la t,t )=R^A(\la;\bU_L,\bU_R), \ \  \ \ \la>0.
\end{equation}
With such a connection, the local wave configuration at the singularity $(0, 0)^\top$ is determined by the  solution of the associated Riemann problem.
For the  GRP solver, it remains to compute the instantaneous time derivative in \eqref{GRP-value}.

For the sake of illustration, we display 
a typical curved wave rarefaction-contact-contact-shock ($R_1+J_2+J_3+S_4$) configuration of the GRP with (non-constant) intermediate fields in \ref{GRP-pic}, together with the associated Riemann problem solution in \ref{RP-pic}. Other wave configurations can be discussed in a similar manner (see Figure \ref{fig:flowchart} for other possible wave configurations). We now proceed to obtain these instantaneous time derivatives of the conservative variables in the intermediate states for two such wave configurations.
\section{Resolution of the generalized Riemann problem for the wave configuration $R_1+J_2+J_3+S_4$}\label{sec: 3}
In this section, we focus on the elementary waves according to the setup of \ref{GRP-pic}, with the 1-rarefaction wave moving to the left and the 4-shock wave moving to the right, separated by two contact discontinuities. In particular, we formulate a uniquely solvable system of linear equations to evaluate the instantaneous time derivatives $(\partial \mathbf{U}/\partial t)^*_{_L}, (\partial \mathbf{U}/\partial t)^*_M$, and $(\partial \mathbf{U}/\partial t)^*_R$ for the wave configuration of \ref{GRP-pic}.  Depending on the specific wave pattern we have then found 
$\left(\dfr{\pt \bU}{\pt t}\right)_*$ in \eqref{GRP-value}.%
\subsection{The resolution of the 1-rarefaction waves}
One of the most significant features of the GRP method is the treatment of the resolution of centered rarefaction waves with characteristic coordinates. By centered rarefaction waves, we mean a curved fan region starting from one single point, which spreads smoothly outwards; see \ref{GRP-pic}. In this section, our main objective is to obtain the time derivatives of the flow variables at the singularity $(0, 0)^{\top}$ across the first rarefaction wave.

The initial data \eqref{data-GRP} can be regarded as perturbation of the Riemann initial data ($\bU_L$, $\bU_R$) such that, in general, a curved rarefaction wave occurs. 
In order to understand how the rarefaction wave expands, we resolve the singularity at the origin
using the technique of nonlinear geometrical optics\cite{ben2006direct}.
Let the space-time sets given by $\{\alpha(x,t) = C_1\}$ and $\{\beta(x,t)=C_2\}$ be the integral curves, respectively, of
\begin{align}
&\frac{dx}{dt}=u+\frac 32v,\ \frac{dx}{dt}=\frac 32 u.\notag
\end{align}
Away from the singularity $(0, 0)^{\top}$, there is a one-to-one correspondence $(x, t)\rw(\alpha, \beta)$ such that
\begin{align}
&\frac{\pt\alpha}{\pt x}\frac{dx}{dt}+\frac{\pt\alpha}{\pt t} = 0,\notag\\
&\frac{\pt\beta}{\pt x}\frac{dx}{dt}+\frac{\pt\beta}{\pt t} = 0.\notag
\end{align}
Regarding the inverse mapping $(\alpha, \beta)\rw(x, t)$, we have
\begin{align}
&\frac{\pt x}{\pt \alpha}=\frac 32u\frac{\pt t}{\pt \alpha},
\frac{\pt x}{\pt \beta}=(u+\frac 32v)\frac{\pt t}{\pt \beta}.\label{characteristic-coordinate}
\end{align}
In the following discussion, we often use both pairs of independent variables $(x, t)$ or $(\alpha, \beta)$.
For example, when we use $(\alpha, \beta)$ as independent variables, we write $(x, t)$ instead of $
(x(\alpha, \beta), t(\alpha, \beta))$, if no confusion is caused.
Thanks to the asymptotics of the GRP to the associated Riemann problem at the singularity
point, we denote by $\beta_L$ the speed of the wave head, by $\beta$ the wave speed
inside the rarefaction wave, and by $\beta^*_L$ the speed of the wave tail. Then we have the following
fact.

\begin{prop}\label{prop-singularity} Consider the curved rarefaction wave associated with $\lambda_1=\frac 32 fb$ and define $\Theta_L(\beta):=\frac{\pt t}{\pt \al}(0,\beta)$.  Then we have
\begin{equation}
\Theta_L(\beta)=\Theta_L(\beta_L)\left(\dfrac{v_L}{v(0,\beta)}\right)^{3/4}.\label{theta}
\end{equation}
\end{prop}

\begin{proof}
From \eqref{characteristic-coordinate},
\begin{align}
&\frac{\pt^2x}{\pt\alpha\pt\beta}=\frac 32\frac{\pt u}{\pt\beta}\frac{\pt t}{\pt\alpha}+\frac 32 u\frac{\pt^2 t}{\pt\alpha\pt\beta},\notag\\
&\frac{\pt^2x}{\pt\alpha\pt\beta}=\frac{\pt}{\pt\alpha}\left(u+\frac 32 v\right)\frac{\pt t}{\pt\beta}+\left(u+\frac 32 v\right)\frac{\pt^2 t}{\pt\alpha\pt\beta}.\notag
\end{align}
Subtracting the second equation from the first one yields
\begin{align}
& \left(\frac 32 v-\frac 12 u\right)\frac{\pt^2 t}{\pt\alpha\pt\beta}=-\frac{\pt t}{\pt\beta}\frac{\pt}{\pt\alpha}\left(u+\frac 32 v\right)+\frac 32\frac{\pt t}{\pt\alpha}\frac{\pt u}{\pt\beta}.\label{t_alphabeta}
\end{align}
We can set $t(0,\beta)=0$ and $\frac 32 \frac{\pt u}{\pt\beta}(0,\beta)=1$. Assume that $\lim_{\alpha\rw 0^+}\frac{\pt}{\pt\alpha}(u+\frac 32 v)(\alpha,\beta)$ is bounded, which implies
\begin{align}
\frac{\pt^2 t}{\pt\alpha\pt\beta}(0,\beta)=\frac{2}{3v-u}\frac{\pt t}{\pt\alpha}(0,\beta).\notag
\end{align}
With $\Theta_L(\beta) = \frac{\pt t}{\pt\alpha}(0,\beta)$, the above equation can be written as
\begin{align}
\frac{\pt}{\pt\beta}\Theta_L(\beta)=\frac{2}{3v-u}\Theta_L(\beta).\notag
\end{align}
This ODE can be solved as
\begin{align}
\frac{\Theta_L(\beta)}{\Theta_L(\beta_L)}=\exp\left(\int_{\beta_L}^{\beta}\frac{2}{3v-u}d\beta\right).
\end{align}
The Riemann variants of the 1-rarefaction are $\xi=f/b$, $\tau=g/q$ and $\eta=(fb+gq)/(gq)^{\frac 14}$.
Notice that if $\alpha\rw 0$, $\beta$ can be selected as $\beta=\frac{x}{t}=\frac{3}{2}u$. Further, differentiating the Riemann invariant $\eta(0,\beta)=(u+v)/(v)^{\frac 14}=\eta_L=\rm{const.}$ w.r.t. $v$ implies
\begin{align}
    \frac{du}{dv}=\frac 14\left(\frac{u}{v}-3\right). 
\end{align}
Thus, \eqref{theta} can be further simplified for $\alpha\rw 0$ as follows,
\begin{align*}
\displaystyle\int_{\beta_L}^{\beta}\dfrac{2}{3v-u} d\beta&= \displaystyle\int_{v_L}^{v(\beta)}\dfrac{3}{3v-u} \dfrac{u-3v}{4v} dv,\\
&=\displaystyle\int_{v_L}^{v(\beta)}-\dfrac{3}{4v} dv= \ln \left(\dfrac{v_{L}}{v(0,\beta)}\right)^{3/4}.
\end{align*}
Then 
\begin{align}
\Theta_L(\beta)=\Theta_L(\beta_L)\left(\dfrac{v_{L}}{v(0,\beta)}\right)^{3/4}.\label{thetal1}
\end{align}
\end{proof}
\begin{remark}
    Proposition \ref{prop-singularity} characterizes how the rarefaction wave expands near the singularity point in terms of the characteristic coordinate $\alpha$. Furthermore, $v=gq$ solely determines the degree of expansion for the curved fan. 
\end{remark}
The total derivatives of $(f, b, g, q)$ are functions of $(\alpha, \beta)$ throughout the centered rarefaction wave. A key ingredient in the resolution of the centered (left) rarefaction wave is the fact that their limiting values as $\alpha \rw 0$, satisfy a system of linear equations. We express this in the following lemma.
\begin{lemma}\label{l-3.1}
Across the 1-rarefaction wave, the trace vector
\[\dfrac{\partial\mathbf{U}}{\partial t}(0, \beta)=\left[
    \dfrac{\partial f}{\partial t},
    \dfrac{\partial  b}{\partial 
 t},
    \dfrac{\partial  g}{\partial 
 t},
    \dfrac{\partial  q}{\partial  t}\right]^\top(0, \beta)\]
    satisfies the following linear equations 
\begin{align}\label{Linear_system_R_2}
    \mathbf{A_{L}}(0, \beta) \dfrac{\partial \mathbf{U}}{\partial t}(0, \beta)= \mathbf{D_{L}}(0, \beta), \quad\forall \beta\in[\beta_{L}, \beta^*_L],
\end{align}
where
\begin{align}
    \mathbf{A_{L}}(0, \beta)=\begin{bmatrix}
        \dfrac{1}{b}(0, \beta) & -\dfrac{\xi}{b}(0, \beta)& 0 &0\\
        0&0& \dfrac{1}{q}(0, \beta)& -\dfrac{\tau}{q}(0, \beta)\\
        b(0, \beta)&f(0, \beta)&\left(\dfrac{3v-u}{4g}\right) (0, \beta)& \left(\dfrac{3v-u}{4q}\right) (0, \beta)  
    \end{bmatrix}, \quad 
\end{align}
and 
\begin{align}
    \mathbf{D_{L}}(0, \beta)=\begin{bmatrix}
        \left(\dfrac{u(0, \beta)}{u_{L}}\right)^{3/2}\left(\dfrac{\partial \xi}{\partial t}\right)_{L}\vspace{0.2 cm}\\
   \dfrac{\Upsilon_L^{\tau}(\beta ;\beta_L)}{2u_L+v_L} \left(\dfrac{\partial \tau}{\partial t}\right)_{L}\left(\dfrac{2u+v}{u-v}\right)(0, \beta)\vspace{0.2 cm}\\
   \dfrac{3v_L-u_L}{(2u_L+3v_L)v_L^{3/4}}\left(\dfrac{\partial \eta}{\partial t}\right)_{L}\left(\dfrac{(2u+3v)v^{3/4}}{3v-u}\right)(0, \beta)
    \end{bmatrix}
\end{align}
with $u=fb$, $v=gq$, $\xi=f/b$, $\tau=g/q$, $\eta=(fb+gq)/(gq)^{1/4}$, and 
\begin{align}\label{eq: gamma_tau}
&\Upsilon_L^{\tau}(\beta;\beta_L)=v^{3/4}(0, \beta)v_L^{-1/2}\bigg(v_L^{-1/4}(u_L+v_L)-2v^{3/4}(0, \beta)\bigg).
\end{align}
\end{lemma}
\begin{proof}
 In order to obtain the linear system for the time derivatives of the conservative variables $\mathbf{U}=(f, b, g, q)^{\top}$ across the rarefaction wave between $\bU_L$ and $\bU^*_L$, we will use the reformulated system \eqref{Rtransformed}. Note that the Riemann invariants $\xi=f/b, \tau=g/q $ and $\eta=(fb+gq)/(gq)^{1/4}$ remain regular across the 1-rarefaction wave, which helps us to obtain the linear system. 
 
 To derive the relation between the time derivatives of $\bU_L$ and $\bU^*_L$, let us take $\xi$ as an example.
The second equation of \eqref{Rtransformed} is rewritten as
\begin{align}
\frac{\pt\xi}{\pt t}+\frac 32 u\frac{\pt\xi}{\pt x} = u\frac{\pt\xi}{\pt x},\ 
\frac{\pt\xi}{\pt t}+\left(u+\frac 32 v\right)\frac{\pt\xi}{\pt x} = \left(\frac 12u+\frac 32v\right)\frac{\pt\xi}{\pt x}.\notag
\end{align}
By changing  to the characteristic coordinates $(x,t)\mapsto (\alpha,\beta)$ for $\xi$, we get
\begin{align}
&\frac{\pt\xi}{\pt\alpha}=\frac{\pt t}{\pt\alpha}\left(\frac{\pt\xi}{\pt t}+\frac 32 u\frac{\pt\xi}{\pt x}\right)=u\frac{\pt t}{\pt\alpha}\frac{\pt\xi}{\pt x},\label{xi1}\\
&\frac{\pt\xi}{\pt\beta}=\frac{\pt t}{\pt\beta}\left(\frac{\pt\xi}{\pt t}+\left(u+\frac 32 v\right)\frac{\pt\xi}{\pt x}\right)=\frac 12 (u+3v)\frac{\pt t}{\pt\beta}\frac{\pt\xi}{\pt x}.\label{xi2}
\end{align}
The second formula is differentiated with respect to $\alpha$. It arrives at 
\begin{align}
\frac{\pt^2\xi}{\pt\alpha\pt\beta}=&\frac 12(u+3v)\frac{\pt^2 t}{\pt\alpha\pt\beta}
\frac{\pt\xi}{\pt x}+\frac 12\frac{\pt t}{\pt \beta}\frac{\pt}{\pt\alpha}\left((u+3v)\frac{\pt\xi}{\pt x}\right).\notag
\end{align}
Letting $\alpha\rw 0^+$ and using the same strategy as for $\Theta_L(\beta)$ in the proof of Proposition \ref{prop-singularity}, an ODE is derived for $\dfrac{\pt \xi}{\pt\alpha}(0,\beta)$ as
\begin{align}
\frac{\pt}{\pt\beta}\left(\frac{\pt \xi}{\pt\alpha}(0,\beta)\right)&=\frac{u+3v}{3v-u}(0,\beta)\frac{\pt t}{\pt\alpha}(0,\beta)\frac{\pt \xi}{\pt x}(0,\beta)\notag\\
&=\frac{u+3v}{u(3v-u)}(0,\beta)\frac{\pt \xi}{\pt\alpha}(0,\beta).\notag
\end{align}
Then we have
\begin{align}
&\frac{\pt \xi}{\pt\alpha}(0,\beta)=\frac{\pt \xi}{\pt\alpha}(0,\beta_L)\Theta_L(\beta)\Upsilon_L^{\xi}(\beta;\beta_L),\notag
\end{align}
where
\begin{align}
\Upsilon_L^{\xi}(\beta;\beta_L)&:=\exp\left(\int_{\beta_L}^{\beta}\frac{1}{u} d\beta\right)=\left(\frac{\beta}{\beta_L}\right)^{\frac 32}=\left(\frac{u(0,\beta)}{u_L}\right)^{\frac 32}.\notag
\end{align}
We go back to the frame $(x, t)$ by using \eqref{xi1} and obtain
\begin{align}
\frac{\pt\xi}{\pt x}(0,\beta)=\frac{\pt}{\pt x}\left(\frac{f}{b}\right)(0,\beta)
&=\frac{u_L}{u(0,\beta)}\left(\frac{\pt\xi}{\pt x}\right)_{L} \Upsilon_L^{\xi}(\beta;\beta_L)=\left(\frac{u(0,\beta)}{u_L}\right)^{1/2}\left(\frac{\pt\xi}{\pt x}\right)_{L}. 
\end{align}
Similar results can be derived for $\tau$ and $\eta$ and lead to
\begin{align}
&\frac{\pt\tau}{\pt x}(0,\beta)=\frac{\pt}{\pt x}\left(\frac{g}{q}\right)(0,\beta)=\frac{1}{(u-v)(0,\beta)}\left(\frac{\pt\tau}{\pt x}\right)_L\Upsilon_L^{\tau}(\beta;\beta_L),\notag\\
&\frac{\pt\eta}{\pt x}(0,\beta)=\frac{\pt}{\pt x}\left(\dfrac{fb+gq}{(gq)^{\frac 14}}\right)(0,\beta)=\frac{(u-3v)_L}{(u-3v)(0,\beta)}\left(\frac{\pt\eta}{\pt x}\right)_{L}\left(\dfrac{v(0, \beta)}{v(0,\beta_{L})}\right)^{3/4},\notag
\end{align}
where
\begin{align}
&\Upsilon_L^{\tau}(\beta;\beta_L):=(u-v)_L\exp\left(\int_{\beta_L}^{\beta}\frac{2}{(u-v)}d\beta\right).\notag
\end{align}
In view of the Riemann invariant $\eta(0,\beta)=(u+v)/(v)^{\frac 14}=\eta_L=\rm{const.}$, we have
\begin{align}
    \frac{du}{dv}=\frac 14\left(\frac{u}{v}-3\right). 
\end{align}
Therefore, we directly compute
\begin{align*}
\displaystyle\int_{\beta_L}^{\beta}\dfrac{2}{u-v}~ d\beta&= \displaystyle\int_{v_L}^{v(\beta)}\dfrac{3}{u-v} \left(\dfrac{u-3v}{4v}\right) dv
=\displaystyle\int_{v_L}^{v(\beta)}\left(\dfrac{3}{4v}-\dfrac{3}{2(u-v)}\right) dv.
\end{align*}
In view of $u+v=\eta_L v^{1/4}$, we have $u-v=\eta_L v^{1/4}-2v$ and therefore
\begin{align*}
\displaystyle\int_{\beta_L}^{\beta}\dfrac{2}{u-v}~ d\beta&=
\displaystyle\int_{v_L}^{v(\beta)}\left(\dfrac{3}{4v}-\dfrac{3}{2v^{1/4}(\eta_L -2v^{3/4})}\right) dv\\
&=\left[\ln v^{3/4}+\ln (\eta_L-2v^{3/4})\right]_{v_L}^{v}\\
&= \ln \left(\dfrac{v^{3/4}(u_L+v_L-2v^{3/4}(v_L)^{1/4})}{(v_L)^{3/4}(u_L-v_L)}\right).
\end{align*}
Thus we have
\begin{align}
&\Upsilon_L^{\tau}(\beta;\beta_L)=v^{3/4}v_L^{-1/2}[v_L^{-1/4}(u_L+v_L)-2v^{3/4}].\notag
\end{align}
From these relations and in view of the system \eqref{Rtransformed}, we obtain the time derivatives of $\xi, \tau$ and $\eta$ given by
\begin{align}
    \left(\dfrac{\partial\xi}{\partial t}\right)(0, \beta)&=-\dfrac{1}{2} \left(u\dfrac{\partial \xi}{\partial x}\right)(0, \beta)=\left(\dfrac{u(0, \beta)}{u_L}\right)^{3/2} \left(\dfrac{\partial \xi}{\partial t}\right)_L,\label{eq: 3.17}\\ 
    \left(\dfrac{\partial\tau}{\partial t}\right)(0, \beta)&=-\dfrac{1}{2}\left((2u+v)\dfrac{\partial \tau}{\partial x}\right)(0, \beta)\nonumber\\
    &=\left(\dfrac{2u+v}{u-v}\right)(0, \beta)\dfrac{\Upsilon_{L}^\tau(\beta; \beta_{L})}{2u_L+v_L}\left(\dfrac{\partial \tau}{\partial t}\right)_L\label{eq: 3.18},\\ 
    \left(\dfrac{\partial\eta}{\partial t}\right)(0, \beta)&=-\dfrac{1}{2}\left((2u+3v)\dfrac{\partial \eta}{\partial x}\right)(0, \beta)\nonumber\\ 
    &=\left(\dfrac{2u+3v}{3v-u}\right)(0, \beta) \frac{3v_L-u_L}{2u_L+3v_L}\left(\dfrac{v(0, \beta)}{v_L}\right)^{3/4}\left(\frac{\pt\eta}{\pt t}\right)_L.\label{eq: 3.19}
\end{align}
Simplifying \eqref{eq: 3.17}, \eqref{eq: 3.18}, and \eqref{eq: 3.19}, we obtain the linear system \eqref{Linear_system_R_2} for the time derivatives of the conservative variables across the 1-rarefaction wave.
\end{proof}

\begin{remark}
    In \eqref{Linear_system_R_2}, $\bU_L$ is known as initial data and $\bU(0,\beta)$ is solved by the associated Riemann problem. $\left(\frac{\partial \xi}{\partial t}\right)_L$, $\left(\frac{\partial \tau}{\partial t}\right)_L$ and $\left(\frac{\pt\eta}{\pt t}\right)_L$ are obtained by the Lax-Wendroff procedure with the initial data of the generalized Riemann problem as follows
    \begin{align}\label{eq: left_derivs}
&\left(\dfrac{\partial \xi}{\partial t}\right)_{L} = -\frac 12 u_L\left(\dfrac{\partial \xi}{\partial x}\right)_{L} = -\frac 12f_L\left(f_L^{\prime}-\frac{f_L}{b_L}b_L^{\prime}\right),\\
&\left(\dfrac{\partial \tau}{\partial t}\right)_{L} = -\left(u_L+\frac 12v_L\right)\left(\dfrac{\partial \tau}{\partial x}\right)_{L} = -\left(u_L+\frac 12v_L\right)\left(\frac{g_L^{\prime}}{q_L}-\frac{g_L}{q_L^2}q_L^{\prime}\right),\\
&\left(\dfrac{\partial \eta}{\partial t}\right)_{L} = -\left(u_L+\frac 32v_L\right)\left(\dfrac{\partial \eta}{\partial x}\right)_{L} = -\left(u_L+\frac 32v_L\right)v_L^{-\frac 14}\left(u_L^{\prime}-\frac 14\left(\frac{u_L}{v_L}+1\right)v_L^{\prime}\right).
\end{align}
\end{remark}
\begin{remark}
The linear system \eqref{Linear_system_R_2} is a linear system for time derivatives across the whole curved 1-rarefaction wave. In particular, for $\beta=\beta^*_L$, we obtain the linear system for $(\partial \mathbf{U}/\partial t)^*_L$.    
\end{remark}
\begin{remark}
    When the $t$-axis ($x = 0$) is located inside the rarefaction fan associated with the characteristic family $\frac 32u$,  we have the sonic case. However, in the state space \eqref{statespace}, $u$ satisfies $u=fb<0$. Therefore, the sonic case will not occur for the first rarefaction wave.
\end{remark}
\subsection{The resolution of contact waves}
Let's take the second contact wave as an example. Considering the Riemann invariants $g$, $q$ and $u=fb$, we have
\begin{align}
&\left(\frac{\partial g}{\partial t}\right)^*_{_M} = \left(\frac{\partial g}{\partial t}\right)^*_{_L},\notag\\
&\left(\frac{\partial q}{\partial t}\right)^*_{_M} = \left(\frac{\partial q}{\partial t}\right)^*_{_L},\notag\\
&\left(\frac{\partial u}{\partial t}\right)^*_{_M} = \left(\frac{\partial u}{\partial t}\right)^*_{_L}.\notag
\end{align}
Thus, the following statement is derived.
\begin{lemma}\label{l-3.2}
Consider a 2-contact wave according to the \ref{GRP-pic}. The vector of the time derivatives of conservative variables 
\[\left(\dfrac{\partial \mathbf{U}}{\partial t}\right)^*_{_M}=\left[
    \left(\dfrac{\partial f}{\partial t}\right)^*_{_M}, ~
    \left(\dfrac{\partial  b}{\partial 
 t}\right)^*_{_M}, ~
    \left(\dfrac{\partial  g}{\partial 
 t}\right)^*_{_M}, ~
    \left(\dfrac{\partial  q}{\partial  t}\right)^*_{_M}
    \right]^{\top}\]
across a 2-contact wave satisfies the following linear system of equations 
\begin{align}\label{Linear_system_C_1}
    \mathbf{A^*_M} \left(\dfrac{\partial \mathbf{U}}{\partial t}\right)^*_{_M}= \mathbf{D^*_L}.
\end{align}
In \eqref{Linear_system_C_1} we denote
\begin{align*}
    \mathbf{A^*_M}=\begin{bmatrix}
        b^*_M, & f^*_M& 0 &0\\
        0&0& g^*_M&0\\
        0&0& 0&q^*_M  
    \end{bmatrix}, ~ \text{and}  ~\mathbf{D^*_L}=\left[
        \left(\dfrac{\pt u}{\pt t}\right)^*_{_L},~
        \left(\dfrac{\pt g}{\pt t}\right)^*_{_L},~
        \left(\dfrac{\pt q}{\pt t}\right)^*_{_L}
    \right]^{\top}.
\end{align*}
Similarly, across the 3-contact wave, the vector of the time derivatives of the conservative variables \[
\left(\dfrac{\partial \mathbf{U}}{\partial t}\right)^*_{_R}=\left[
    \left(\dfrac{\partial f}{\partial t}\right)^*_{_R}, ~
    \left(\dfrac{\partial  b}{\partial 
 t}\right)^*_{_R}, ~
    \left(\dfrac{\partial  g}{\partial 
 t}\right)^*_{_R}, ~
    \left(\dfrac{\partial  q}{\partial  t}\right)^*_{_R}
    \right]^{\top}
    \]
    satisfies the following linear system of equations 
\begin{align}\label{Linear_system_C_2}
    \mathbf{A^*_R}\left(\dfrac{\partial \mathbf{U}}{\partial t}\right)^*_{_R}= \mathbf{D^*_M},
\end{align}
where
\begin{align*}
    \mathbf{A^*_R}=\begin{bmatrix}
        f^*_R, & 0& 0 &0\\
        0& b^*_R&0& 0\\
        0&0& q^*_R  &g^*_R  
    \end{bmatrix}, ~\text{and }\mathbf{D^*_M}=\left[
        \left(\dfrac{\pt f}{\pt t}\right)^*_{_M},~
        \left(\dfrac{\pt b}{\pt t}\right)^*_{_M},~
        \left(\dfrac{\pt v}{\pt t}\right)^*_{_M}
    \right]^{\top}.
\end{align*}
\end{lemma}
\subsection{The resolution of the 4-shock wave}
In this section, we follow the idea of \cite{ben2006direct} to resolve the shock at the origin. The key part is to apply the Rankine-Hugoniot condition of the shock wave \eqref{eq: 24b} and use the continuity property of solutions adjacent to the shock front. 

Inherently, we make
differentiation along the shock trajectory $x = x(t)$. The Rankine-Hugoniot condition of $4$-shock wave implies that
\begin{align}
    \Gamma_1:=f-f_R=0,~\Gamma_2:=b-b_R=0,~\Gamma_3:=\dfrac{{g}}{{q}}-\dfrac{{g_R}}{{q_R}}=0,
\end{align}
which is also equivalent to
\begin{align}
    \Gamma_1^{\prime}:=\xi-\xi_R=0,~\Gamma_2^{\prime}:=u-u_R=0,~\Gamma_3:=\dfrac{{g}}{{q}}-\dfrac{{g_R}}{{q_R}}=0.
\end{align}
Differentiating along $\Gamma_1^{\prime}$, $\Gamma_2^{\prime}$ and $\Gamma_3$, we have
\begin{align}
&\frac{D_{\sigma_4}\Gamma_1^{\prime}}{dt}=\left(\frac{\pt \xi}{\pt t}\right)^*_{_R}+\sigma_4\left(\frac{\pt \xi}{\pt x}\right)^*_{_R}-\left(\frac{\pt \xi}{\pt t}\right)_R-\sigma_4\left(\frac{\pt \xi}{\pt x}\right)_R=0,\label{4s-1}\\
&\frac{D_{\sigma_4}\Gamma_2^{\prime}}{dt}=\left(\frac{\pt u}{\pt t}\right)^*_{_R}+\sigma_4\left(\frac{\pt u}{\pt x}\right)^*_{_R}-\left(\frac{\pt u}{\pt t}\right)_R-\sigma_4\left(\frac{\pt u}{\pt x}\right)_R=0,\label{4s-2}\\
&\frac{D_{\sigma_4}\Gamma_3}{dt}=\left(\frac{\pt \tau}{\pt t}\right)^*_{_R}+\sigma_4\left(\frac{\pt \tau}{\pt x}\right)^*_{_R}-\left(\frac{\pt \tau}{\pt t}\right)_R-\sigma_4\left(\frac{\pt \tau}{\pt x}\right)_R=0,\label{4s-3}
\end{align}
where $\frac{D_{\sigma_4}}{dt}=\frac{\pt }{\pt t}+\sigma_4\frac{\pt}{\pt x}$ and $\sigma_4$ is the shock speed of the 4-shock wave as $t\rw 0$. 
The key point here is based on the Lax-Wendorff procedure by using the equations that express the unknown temporal derivatives $\left(\frac{\pt \omega}{\pt t}\right)_R$ by known spatial derivatives $\left(\frac{\pt \omega}{\pt x}\right)_R$ and the unknown $(\frac{\pt \omega}{\pt x})^*_R$ by $(\frac{\pt \omega}{\pt t})^*_R$, $\omega\in \{\xi,u,\tau\}$. Then a linear system for $(\frac{\pt f}{\pt t})^*_R$, $(\frac{\pt b}{\pt t})^*_R$ and $(\frac{\pt \tau}{\pt t})^*_R$ is derived. Precisely, we have the following lemma.
\begin{lemma}\label{l-3.3}
Across the 4-shock wave, the vector of the time derivatives of conservative variables 
\[\left(\dfrac{\partial \mathbf{U}}{\partial t}\right)^*_{_R}=\left[
    \left(\dfrac{\partial f}{\partial t}\right)^*_{_R}, ~
    \left(\dfrac{\partial  b}{\partial 
 t}\right)^*_{_R}, ~
    \left(\dfrac{\partial  g}{\partial 
 t}\right)^*_{_R}, ~
    \left(\dfrac{\partial  q}{\partial  t}\right)^*_{_R}
    \right]^{\top}
    \]
    satisfies a linear system of equations  given by
\begin{align}\label{Linear_system_S_1}
    \mathbf{A_R} \left(\dfrac{\partial \mathbf{U}}{\partial t}\right)^*_{_R}= \mathbf{D_{R}}.
\end{align}
Here we have
\begin{align}
    \mathbf{A_R}=\begin{bmatrix}
        b_{R} & f_{R}& 0 &0\\
        \dfrac{1}{b_R}&-\dfrac{f_R}{(b_R)^2}& 0 & 0\\
        0&0& \dfrac{1}{q^*_R}&-\dfrac{g^*_R}{(q^*_R)^2}  
    \end{bmatrix}, \quad 
\end{align}
and 
\begin{align}
    \mathbf{D_{R}}=
       \left[-\dfrac{3u_R}{2}u^{\prime}_R, \quad
       -\dfrac{u_R}{2}\xi^{\prime}_R, \quad
        \frac{(2 u^*_R+v^*_R)}{2}\left(\dfrac{v^*_R}{v_R}\right)^{1/2}\tau_R^{\prime}
    \right]^\top.
\end{align}  
\end{lemma}
\begin{proof}
In view of \eqref{4s-1}-\eqref{4s-2} with $t\rw 0$ and \eqref{Rtransformed}, we have
\begin{align*}
     \left(\frac{\pt \xi}{\pt t}\right)^*_{_R}\left(1-\dfrac{2\sigma_4}{u^*_R}\right)=-\dfrac{u_R}{2}\left(1-\dfrac{2\sigma_4}{u_R}\right)\left(\frac{\pt \xi}{\pt x}\right)_R
\end{align*}
and 
\begin{align*}
     \left(\frac{\pt u}{\pt t}\right)^*_{_R}\left(1-\dfrac{2\sigma_4}{3u^*_R}\right)=-\dfrac{3u_R}{2}\left(1-\dfrac{2\sigma_4}{3u_R}\right)\left(\frac{\pt u}{\pt x}\right)_R.
\end{align*}
Since $u^*_R=u_R$ and $2\sigma_4\neq u_R$ across 4-shock wave, it will be simplified as
\begin{align}
     \left(\frac{\pt \xi}{\pt t}\right)^*_R&=-\dfrac{u_R}{2}\xi^{\prime}_R,\label{xirelations}\\
     \left(\frac{\pt u}{\pt t}\right)^*_R&=-\dfrac{3u_R}{2}u^{\prime}_R. \label{urelations}
\end{align}
where $\xi^{\prime}_R = \frac{1}{b_R}f_R^{\prime}-\frac{f_R}{b_R^2}b_R^{\prime}$, $u^{\prime}_R=b_Rf_R^{\prime}+b_R^{\prime}f_R$.

From the formula \eqref{4s-3}, we have
\begin{align}
\left(\frac{\pt \tau}{\pt t}\right)^*_R-\frac{2\sigma_4}{2u^*_R+v^*_R}\left(\frac{\pt \tau}{\pt t}\right)^*_R=-\left(u_R+\frac 12v_R\right)\left(\frac{\pt \tau}{\pt x}\right)_R+\sigma_4\left(\frac{\pt \tau}{\pt x}\right)_R,\notag
\end{align}
that is
\begin{align}\label{taurelations}
\left(\frac{\pt \tau}{\pt t}\right)^*_R=\dfrac{(2u^*_R+v^*_R)(\sigma_4-u_R-\frac 12v_R)}{2(2u^*_R+v^*_R-2\sigma_4)}\tau_R^{\prime}=\frac{(2 u^*_R+v^*_R)}{2}\left(\dfrac{v^*_R}{v_R}\right)^{1/2}\tau_R^{\prime}.
\end{align}
where $\tau_R^{\prime}=\frac{1}{q_R}g_R^{\prime}-\frac{g_R}{q_R^2}q_R^{\prime}$.

In view of \eqref{xirelations}, \eqref{urelations} and \eqref{taurelations}, we obtain the linear system \eqref{Linear_system_S_1}.
\end{proof}
\begin{remark}
    In \eqref{Linear_system_S_1}, $\bU_R$ and $\bU_R^{\prime}$ are the initial data of the generalized Riemann problem and $\bU^*_R$ is solved by the associated Riemann problem. Specifically, the first two equations in \eqref{Linear_system_S_1} can be directly solved for $(\frac{\pt f}{\pt t})^*_R$ and $(\frac{\pt b}{\pt t})^*_R$. The last equation will be equipped with \eqref{Linear_system_R_2}, \eqref{Linear_system_C_1} and \eqref{Linear_system_C_2} for solving the remaining unknowns. 
\end{remark}
\subsection{Time derivatives of solutions of \eqref{eq: Main_system} at the singularity $(0, 0)^\top$ for the wave configuration $R_1+J_2+J_3+S_4$}\label{sec: R+J+J+S}
In this section, we use the results of the previous section in order to calculate the instantaneous values $(\partial \mathbf{U}/\partial t)^*_L, (\partial \mathbf{U}/\partial t)_M^*$, and $(\partial \mathbf{U}/\partial t)^*_R$ for the wave configuration of \ref{GRP-pic}. In total, there are 12 unknown time derivatives. However, from \eqref{Linear_system_R_2}, \eqref{Linear_system_C_1}, \eqref{Linear_system_C_2} and \eqref{Linear_system_S_1}, we observe that $\left(\dfrac{\partial f}{\partial t}\right)^*_R$ and $\left(\dfrac{\partial b}{\partial t}\right)^*_R$ are directly solved by \eqref{Linear_system_S_1} while $\left(\dfrac{\partial g}{\partial t}\right)_M^*$, $\left(\dfrac{\partial q}{\partial t}\right)_M^*$, $\left(\dfrac{\partial f}{\partial t}\right)_M^*$, and $\left(\dfrac{\partial b}{\partial t}\right)_M^*$ satisfy:
\begin{align}\label{eq: derivative_relations}
\begin{cases}
\left(\dfrac{\partial g}{\partial t}\right)_M^*= \left(\dfrac{\partial g}{\partial t}\right)^*_L, \left(\dfrac{\partial q}{\partial t}\right)_M^*=\left(\dfrac{\partial q}{\partial t}\right)^*_L
,\vspace{0.2 cm}\\
\left(\dfrac{\partial f}{\partial t}\right)_M^*=\left(\dfrac{\partial f}{\partial t}\right)^*_R
, \left(\dfrac{\partial b}{\partial t}\right)_M^*=\left(\dfrac{\partial b}{\partial t}\right)^*_R.
\end{cases}
\end{align}
There remain six unknowns that need to be computed which are
\begin{align*}
    \dfrac{\partial \hat{\bU}}{\partial t}=\left[\left(\dfrac{\partial f}{\partial t}\right)^*_L,~\left(\dfrac{\partial b}{\partial t}\right)^*_L,~\left(\dfrac{\partial g}{\partial t}\right)^*_L,~\left(\dfrac{\partial q}{\partial t}\right)^*_L,~\left(\dfrac{\partial g}{\partial t}\right)^*_R, ~~\left(\dfrac{\partial q}{\partial t}\right)^*_R\right]^\top.
\end{align*}
Using the linear relations obtained in the Lemmas \ref{l-3.1}, \ref{l-3.2} and \ref{l-3.3}, we can obtain a system of six linear equations that provide these remaining time derivatives.
\begin{theorem}[Time-derivatives of the conservative variables for the wave configuration of \ref{GRP-pic}]\label{non-sonic-case}
\rm
The trace values $(\partial \mathbf{U}/\partial t)^*_L, (\partial \mathbf{U}/\partial t)_M^*$ and $(\partial \mathbf{U}/\partial t)^*_R$ in \eqref{GRP-value} are obtained by \eqref{eq: derivative_relations} and solving an invertible system of linear equations for $\dfrac{\partial\hat{\bU}}{\partial t}$ given by
\begin{align}\label{main_linear_system1}
    \mathbf{A(U^*)}\dfrac{\partial \hat{\bU}}{\partial t}= \mathbf{D(U_{L}, U_{R},U^*,U_L^{\prime},U_R^{\prime})}.
\end{align}
The coefficient matrix $\mathbf{A}$ and the vector $\mathbf{D(U_L, U_R,U^*,U_L^{\prime},U_R^{\prime})}$ only depend on the initial data or the intermediate states of the associated Riemann problem \eqref{RP-data}, and are given by

\begin{align}
    \mathbf{A(U^*)}=\begin{bmatrix}
        b^*_L & f^*_L &0&0&0&0\\
        \dfrac{1}{b^*_L}&-\dfrac{f^*_L}{(b^*_L)^2}&0&0&0&0\\
        b^*_L & f^*_L & \dfrac{3v^*_L-u^*_L}{4g^*_L}& \dfrac{3v^*_L-u^*_L}{4q^*_L}&0&0\\
        0 & 0 & \dfrac{1}{q^*_L}&-\dfrac{g^*_L}{(q^*_L)^2}&0&0\\
        0 & 0 &-q^*_L& -g^*_L& q^*_R& g^*_R\\
        0 & 0 & 0 & 0&\dfrac{1}{q^*_R}&-\dfrac{g^*_R}{(q^*_R)^2}\\
    \end{bmatrix},
\end{align}
 and
 \begin{align}
\mathbf{D(U_{L}, U_{R},U^*,U_L^{\prime},U_R^{\prime}))}=\begin{bmatrix}
    \left(\dfrac{\pt u}{\pt t}\right)_R\vspace{0.2 cm}\\
    \left(\dfrac{u^*_L}{u_{L}}\right)^{3/2}\left(\dfrac{\partial \xi}{\partial t}\right)_{L}\vspace{0.2 cm}\\
   \dfrac{3v_L-u_L}{(2u_L+3v_L)v_L^{3/4}}\left(\dfrac{\partial \eta}{\partial t}\right)_{L}\dfrac{(2u^*_L+3v^*_L)(v^*_L)^{3/4}}{3v^*_L-u^*_L}\vspace{0.2 cm}\\ 
   \dfrac{\Upsilon_L^{\tau}(\beta^*_L;\beta_L)}{2u_L+v_L} \left(\dfrac{\partial \tau}{\partial t}\right)_{L}\dfrac{2u^*_L+v^*_L}{u^*_L-v^*_L}\vspace{0.2 cm}\\
    0\vspace{0.2 cm}\\
 \dfrac{(2 u^*_R+v^*_R)}{2u_R+v_R}\left(\dfrac{v^*_R}{v_R}\right)^{1/2}\left(\dfrac{\pt \tau}{\pt t}\right)_R
\end{bmatrix} 
\end{align}
with 
\[
\Upsilon_L^{\tau}(\beta^*_L;\beta_{L})=(v^*_L)^{3/4}v_L^{-1/2}\bigg(v_L^{-1/4}(u_L+v_L)-2(v^*_L)^{3/4}\bigg).\notag
\]
\end{theorem}
\begin{proof}
A simple deduction using Lemmas \ref{l-3.1}, \ref{l-3.2}, and \ref{l-3.3} leads to the system \eqref{main_linear_system1}. 
The linear system \eqref{main_linear_system1} is uniquely solvable as the matrix $\mathbf{A(U)}$ is invertible. This can be easily seen by its block structure. Each subblock of $\mathbf{A(U)}$ has a nonzero determinant for $\mathbf{U}\in \mathcal{U}$.
\end{proof}
\begin{remark}
Due to the block nature of the matrix $\mathbf{A(U)}$, solving the linear system \eqref{main_linear_system1} becomes cheap. One needs to solve each $(2\times 2)$-subsystem, which can be done explicitly.
\end{remark}
\subsection{The acoustic case}\label{sec: acoustic}
The acoustic case is the linear simplification of the generalized problem with following initial data
$\mathbf{U}_L=\mathbf{U}_R$ and $\mathbf{U}_L'\neq \mathbf{U}_R'$. The nomenclature for this case is borrowed from the GRP methods for compressible Euler equations; see e.g. \cite{ben2006direct}. In this case, only contact waves emanate from the singularity $(0, 0)^\top$, and therefore, the scheme becomes simple. Using Theorem \ref{non-sonic-case}, it is easy to obtain the acoustic case.  We state this in the following theorem.

\begin{theorem}[Acoustic case]\label{acoustic}
    When $\mathbf{U}_L=\mathbf{U}_R=\bU_0$ and $\mathbf{U}_L'\neq \mathbf{U}_R'$, the time derivatives $\dfrac{\partial\mathbf{U^*}}{\partial t}=\bigg[
    \left(\dfrac{\partial f}{\partial t}\right)^*_L,
   \left(\dfrac{\partial b}{\partial t}\right)^*_L,
     \left(\dfrac{\partial g}{\partial t}\right)^*_L,
    \left(\dfrac{\partial q}{\partial t}\right)^*_L,
    \left(\dfrac{\partial g}{\partial t}\right)^*_R,
   \left(\dfrac{\partial q}{\partial t}\right)^*_R\bigg]^{\top}$ can be obtained by solving the following system of linear equations:
 \begin{align}
     \mathbf{A(U_0)} \dfrac{\partial \mathbf{U^*}}{\partial t}= \mathbf{D}(\bU_0,\mathbf{U}_L',\mathbf{U}_R'),
 \end{align}   
where the coefficient matrix $\mathbf{A}$ and the vector $\mathbf{D}$ are given by
\begin{align}
    \mathbf{A(U_0)}=\begin{bmatrix}
        b_{0}& f_{0}&0&0&0&0\\
        \dfrac{1}{b_{0}}&-\dfrac{f_{0}}{(b_{0})^2}&0 &0&0&0\\
        0 & 0 & q_{0}& g_{0}&0&0\\
        0 & 0 & \dfrac{1}{q_{0}}&-\dfrac{g_{0}}{(q_{0})^2} &0&0\\
        0 & 0 &-q_{0}& -g_{0}& q_{0}& g_{0}\\
        0 & 0 & 0 & 0&\dfrac{1}{q_{0}}&-\dfrac{g_{0}}{(q_{0})^2}\\
    \end{bmatrix},
\end{align}
 and
 \begin{align}
\mathbf{D(U_{0},U_L',U_R')}=\begin{bmatrix}
   -\dfrac{3}{2}u_0 u_R'\vspace{0.2 cm}\\
   -\dfrac{1}{2}u_0 \xi_L'\vspace{0.2 cm}\\
    \dfrac{4v_0}{3v_0-u_0}\left(\dfrac{3}{2}u_0 u_R'-(u_0+\frac{3}{2}v_0)\left(u_L'+v_L'\left(\dfrac 34-\dfrac{u_0}{4v_0}\right)\right)\right)\vspace{0.2 cm}\\    
    -\dfrac{1}{2}(2u_0+v_0) \tau_L'\vspace{0.2 cm} \\
    0\vspace{0.2 cm}\\
-\dfrac{1}{2}(2u_0+v_0) \tau_R'
\end{bmatrix}.
\end{align}

\end{theorem}
We have covered all possible cases to compute time derivatives for the wave configuration of \ref{GRP-pic}. We now proceed to discuss the case when the first wave is a shock wave and the fourth wave is a rarefaction wave.
\section{Resolution of the generalized Riemann problem for the wave configuration $S_1+J_2+J_3+R_4$}\label{sec: 4}
In this section, we consider the wave configuration when the first wave is a shock and the fourth wave is a rarefaction wave connected via two contact discontinuities ($S_1+J_2+J_3+R_4$). Note that the fourth characteristic field of the system \eqref{eq: Main_system} is actually a Temple field, as discussed in Section \ref{sec: 2.3}. Therefore, the construction of the GRP method for the other two wave configurations ($S_1+J_2+J_3+S_4$ and $R_1+J_2+J_3+R_4$) is similar and doesn't need a separate discussion. 

In what follows, we consider the wave configuration $S_1+J_2+J_3+R_4$ and obtain time derivatives of conservative variables in the intermediate states by solving a system of linear equations. 
\subsection{{Resolution of the 1-shock wave}}
Let $x=x(t)$ be the shock trajectory which is associated with the $(\lambda_1=\frac 32fb)$ first-characteristic family and assume that it propagates with the shock speed $\sigma_1=x^{\prime}(t)$. $\bU_L$ and $\bU^*_L$ are used to denote the pre-shock and post-shock values of $\bU$, respectively. Along this shock, the Rankine-Hugoniot relations simplify to
\begin{align}
\Xi_1^2 &= \xi_L-\xi^*_L=0,\ \xi = \frac{f}{b},\label{2s-rh1}\\
\Xi_2^2 &= \tau_L-\tau^*_L = 0,\ \tau = \frac{g}{q},\label{2s_rh2}\\
\Xi_3^2 &= q^*_L-\tilde{\Phi}(f_L,b_L,f^*_L,b^*_L,g_L,q_L,g^*_L)=0,\label{2s_rh3}\\
\tilde{\Phi} := &\left(\frac{g_L}{g^*_L}\right)^2q_L + \frac{1}{(g^*_L)^2}(f_Lb_Lg_L+f_Lb_Lg^*_L-f^*_Lb_Lg_L\notag\\
&+f^*_Lb_Lg^*_L-f^*_Lb^*_Lg_L-f^*_Lb^*_Lg^*_L),\label{2s_rh4}
\end{align}
where \eqref{2s_rh3}-\eqref{2s_rh4} are deduced by \eqref{S_2def}.

Using the nonconservative form \eqref{eq: 3.6}, it is easy to obtain
\begin{align}
&\frac{\pt f}{\pt x}=\frac{2}{3b^2}\frac{\pt b}{\pt t}-\frac{4}{3fb}\frac{\pt f}{\pt t},\label{derivsofU5}\\
&\frac{\pt b}{\pt x}=\frac{2}{3f^2}\frac{\pt f}{\pt t}-\frac{4}{3fb}\frac{\pt b}{\pt t},\label{derivsofU6}\\
&\frac{\pt g}{\pt x}=\vartheta_2 g^2\frac{\pt q}{\pt t}-\vartheta_3\frac{\pt g}{\pt t}+\frac{2}{3u}\vartheta_1 g\frac{\pt (fb)}{\pt t},\label{derivsofU7}\\
&\frac{\pt q}{\pt x}=\vartheta_2 q^2\frac{\pt g}{\pt t}-\vartheta_3\frac{\pt q}{\pt t}+\frac{2}{3u}\vartheta_1 q\frac{\pt (fb)}{\pt t},\label{derivsofU8}
\end{align}
where $\vartheta_1=\dfrac{2}{2u+3v}$, $\vartheta_2=\dfrac{2}{(2u+3v)(2u+v)}$,$\vartheta_3=\dfrac{4(u+v)}{(2u+3v)(2u+v)}$. Then we have the following result.
\begin{lemma}\label{l-A.1}
Across the 1-shock wave, the trace values $\left(\dfrac{\partial\mathbf{U}}{\partial t}\right)^*_L$ satisfy the following linear equations    
\begin{align}\label{Linear_system_Shock}
\mathbf{A^*_L}\left(\dfrac{\partial \mathbf{U}}{\partial t}\right)^*_L= \mathbf{D_{L}}.
\end{align}
Here
\begin{align}
    \mathbf{A^*_L}&=\begin{bmatrix}
        \dfrac{1}{b^*_L} & -\dfrac{f^*_L}{(b^*_L)^2}& 0 &0\\
        0&0& \dfrac{1}{q^*_L}& -\dfrac{g^*_L}{(q^*_L)^2}\\
        \mathcal{A}^*_L&\mathcal{B}^*_L&\mathcal{C}^*_L & \mathcal{D}^*_L  
    \end{bmatrix}, \quad \\
    \mathbf{D_{L}}&=\begin{bmatrix}
        \left(\dfrac{u^*_L}{u_L}\right)^{3/2}\left(\dfrac{\pt\xi}{\pt t}\right)_{L}\vspace{0.2 cm}\\
   \dfrac{(2u^*_L+v^*_L)(2u_L+v_L-2\sigma_1)}{(2u_L+v_L)(2u^*_L+v^*_L-2\sigma_1)}\left(\dfrac{\partial \tau}{\partial t}\right)_{L}\vspace{0.2 cm}\\
   \mathcal{A}_L\left(\dfrac{\partial f}{\partial t}\right)_L+\mathcal{B}_L\left(\dfrac{\partial b}{\partial t}\right)_L+\mathcal{C}_L\left(\dfrac{\partial g}{\partial t}\right)_L+\mathcal{D}_L\left(\dfrac{\partial q}{\partial t}\right)_L
    \end{bmatrix},
\end{align}
with 
\begin{align}\label{shock-coefficients_1}
\begin{cases}
\mathcal{A}^*_L&=-\left[\left(1-\dfrac{4\sigma_1}{3u^*_L}\right)\Delta_5+\dfrac{2\sigma_1}{3(f^*_L)^2}\Delta_6+\dfrac{2\sigma_1(\vartheta_1)^*_L}{3f^*_L}[q^*_L(g^*_L)^2+{2}\Delta_1]\right],\vspace{0.2cm}\\
\mathcal{B}^*_L&=-\left[\left(1-\dfrac{4\sigma_1}{3u^*_L}\right)\Delta_6+\dfrac{2\sigma_1}{3(b^*_L)^2}\Delta_5+\dfrac{2\sigma_1(\vartheta_1)^*_L}{3b^*_L}[q^*_L(g^*_L)^2+{2}\Delta_1]\right],\vspace{0.2cm}\\
\mathcal{C}^*_L&=\left[\sigma_1 (\vartheta_2)^*_L(g^*_L q^*_L)^2+\dfrac{2\Delta_1}{g^*_L}(1-\sigma_1 (\vartheta_3)^*_L)\right],\vspace{0.2cm}\\
    \mathcal{D}^*_L&=\left[(g^*_L)^2(1-\sigma_1(\vartheta_3)^*_L)+2g^*_L\sigma_1 \Delta_1(\vartheta_2)^*_L\right], 
\end{cases}
\end{align}
\begin{align}\label{shock-coefficients-2}
\begin{cases}
\mathcal{A}_{L}&=-\left[\Delta_2\left(1-\dfrac{4\sigma_1}{3u_L}\right)+\dfrac{2\sigma_1}{3(f_L)^2}\Delta_3+\dfrac{2\sigma_1(\vartheta_1)_L}{3f_L}(q_L(g_L)^2+\Delta_4 g_L)\right],\vspace{0.2cm}\\
\mathcal{B}_{L}&=-\left[\dfrac{2\sigma_1}{3(b_L)^2}\Delta_2+\Delta_3(1-\dfrac{4\sigma_1}{3u_L})+\dfrac{2\sigma_1(\vartheta_1)_L}{3b_L}(q_L(g_L)^2+\Delta_4g_L)\right],\vspace{0.2cm}\\
\mathcal{C}_{L}&=\left[(g_Lq_L)^2\sigma_1(\vartheta_2)_L+(1-\sigma_1(\vartheta_3)_L)\Delta_3\right],\vspace{0.2cm}\\
    \mathcal{D}_{L}&=(g_L)^2\left[1-\sigma_1(\vartheta_3)_L+\sigma_1(\vartheta_2)_L\Delta_4\right]. 
\end{cases}
\end{align}   
In \eqref{shock-coefficients_1} and \eqref{shock-coefficients-2}, the coefficients are given by
\begin{align}
&\Delta_1 = v_Lg_L+(u_L-u^*_L)\left(g_L+\dfrac 12g^*_L\right)-f^*_Lb_L\left(g_L-\dfrac 12g^*_L\right),\notag\\
&\Delta_2 = b_L(g^*_L+g_L),\notag\\
&\Delta_3 = f_L(g^*_L+g_L)+f^*_L(g^*_L-g_L),\notag\\
&\Delta_4 = 2v_L+u_L-u^*_L-f^*_Lb_L,\notag\\
&\Delta_5 = b^*_L(g_L+g^*_L)+b_L(g_L-g^*_L),\notag\\
&\Delta_6 = f^*_L(g_L+g^*_L).\notag
\end{align}
\end{lemma}
\begin{proof}
The proof of this lemma is similar as Lemma \ref{l-3.3}.
We just need to take the directional derivative along the shock trajectory $x=x(t)$, 
\begin{align}
&\frac{D_{\sigma_1}\Xi_1}{dt}=\left(\frac{\pt}{\pt t}+\sigma_1\frac{\pt}{\pt x}\right)\Xi_1 = 0,\label{2s-1}\\
&\frac{D_{\sigma_1}\Xi_2}{dt}=\left(\frac{\pt}{\pt t}+\sigma_1\frac{\pt}{\pt x}\right)\Xi_2 = 0,\label{2s-2}\\
&\frac{D_{\sigma_1}\Xi_3}{dt}=\left(\frac{\pt}{\pt t}+\sigma_1\frac{\pt}{\pt x}\right)\Xi_3 = 0,\label{2s-3}
\end{align}
and apply the Lax-Wendroff procedure by using the continuity property of solutions adjacent to the shock front.
Since $\xi$ and $\tau$ satisfy \eqref{Rtransformed}, we can use that
\begin{align}
&\frac{\pt \xi}{\pt t}=-\frac 12 u\frac{\pt \xi}{\pt x},\ \frac{\pt \xi}{\pt x}=-\frac{2}{u}\frac{\pt \xi}{\pt t},\notag\\
&\frac{\pt \tau}{\pt t}=-\left(u+\frac 12 v\right)\frac{\pt \tau}{\pt x},\
\frac{\pt \tau}{\pt x}=-\frac{2}{2u+v}\frac{\pt \tau}{\pt t}.\notag
\end{align}
For the first and second formula \eqref{2s-1}-\eqref{2s-2}, we have
\begin{align}
&\left(\frac{\pt\xi}{\pt t}\right)^*_{_L}-\frac{2\sigma_1}{u^*_L}\left(\frac{\pt\xi}{\pt t}\right)^*_{_L} = \left(\frac{\pt\xi}{\pt t}\right)_L-\frac{2\sigma_1}{u_L}\left(\frac{\pt\xi}{\pt t}\right)_L,\notag\\
&\left(\frac{\pt\tau}{\pt t}\right)^*_{_L}-\frac{2\sigma_1}{2u^*_L+v^*_L}\left(\frac{\pt\tau}{\pt t}\right)^*_{_L} = \left(\frac{\pt\tau}{\pt t}\right)_L-\frac{2\sigma_1}{2u_L+v_L}\left(\frac{\pt\tau}{\pt t}\right)_L,\notag
\end{align}
that is
\begin{align}
&\left(\frac{\pt\xi}{\pt t}\right)^*_{_L} = \frac{u^*_L(u_L-2\sigma_1)}{u_L(u^*_L-2\sigma_1)}\left(\frac{\pt\xi}{\pt t}\right)_L=\left(\dfrac{u^*_L}{u_L}\right)^{3/2}\left(\frac{\pt\xi}{\pt t}\right)_{L},\label{xi_shock}\\
&\left(\frac{\pt\tau}{\pt t}\right)^*_{_L}= \frac{(2u^*_L+v_L^*)(2u_L+v_L-2\sigma_1)}{(2u_L+v_L)(2u^*_L+v_L^*-2\sigma_1)}\left(\frac{\pt\tau}{\pt t}\right)_L.\label{tau_shock}
\end{align}
For \eqref{2s-3}, the idea is similar, but the formula will be more intricate.
\begin{align}
\frac{D_{\sigma_1}\Xi_3}{Dt} = &\left(\frac{\pt q}{\pt t}\right)^*_{_L}+\sigma_1\left(\frac{\pt q}{\pt x}\right)^*_{_L}-\left(\frac{g_L}{g^*_L}\right)^2\left(\left(\frac{\pt q}{\pt t}\right)_L+\sigma_1\left(\frac{\pt q}{\pt x}\right)_L\right)\notag\\
&-\frac{2v_L}{(g^*_L)^2}\left(\left(\frac{\pt g}{\pt t}\right)_L+\sigma_1\left(\frac{\pt g}{\pt x}\right)_L\right)+\frac{2v_Lg_L}{(g^*_L)^3}\left(\left(\frac{\pt g}{\pt t}\right)^*_{_L}+\sigma_1\left(\frac{\pt g}{\pt x}\right)^*_{_L}\right)\notag\\
&+\frac{2}{(g^*_L)^3}((u_L-u^*_L)(g_L+g^*_L)-f^*_Lb_L(g_L-g^*_L))\left(\left(\frac{\pt g}{\pt t}\right)^*_{_L}+\sigma_1\left(\frac{\pt g}{\pt x}\right)^*_{_L}\right)\notag\\
&-\frac{b_L(g^*_L+g_L)}{(g^*_L)^2}\left(\left(\frac{\pt f}{\pt t}\right)_L+\sigma_1\left(\frac{\pt f}{\pt x}\right)_L\right)\notag\\
&-\frac{f_L(g_L+g^*_L)-f^*_L(g_L-g^*_L)}{(g^*_L)^2}\left(\left(\frac{\pt b}{\pt t}\right)_L+\sigma_1\left(\frac{\pt b}{\pt x}\right)_L\right)\notag\\
&+\frac{b^*_L(g_L+g^*_L)+b_L(g_L-g^*_L)}{(g^*_L)^2}\left(\left(\frac{\pt f}{\pt t}\right)^*_{_L}+\sigma_1\left(\frac{\pt f}{\pt x}\right)^*_{_L}\right)\notag\\
&+\frac{f^*_L(g_L+g^*_L)}{(g^*_L)^2}\left(\left(\frac{\pt b}{\pt t}\right)^*_{_L}+\sigma_1\left(\frac{\pt b}{\pt x}\right)^*_{_L}\right)\notag\\
&-\frac{u_L-u^*_L-f^*_Lb_L}{(g^*_L)^2}\left(\left(\frac{\pt g}{\pt t}\right)_L+\sigma_1\left(\frac{\pt g}{\pt x}\right)_L\right)\notag\\
&-\frac{u_L-u^*_L+f^*_Lb_L}{(g^*_L)^2}\left(\left(\frac{\pt g}{\pt t}\right)^*_{_L}+\sigma_1\left(\frac{\pt g}{\pt x}\right)^*_{_L}\right)=0.\notag
\end{align}
By rearranging these terms, we have
\begin{align}
&~(g^*_L)^2\left[\left(\frac{\pt q}{\pt t}\right)^*_{_L}+\sigma_1\left(\frac{\pt q}{\pt x}\right)^*_{_L}\right]+\frac{2}{g^*_L}\Delta_1\left[\left(\frac{\pt g}{\pt t}\right)^*_{_L}+\sigma_1\left(\frac{\pt g}{\pt x}\right)^*_{_L}\right]\notag\\
&-\Delta_5\left[\left(\frac{\pt f}{\pt t}\right)^*_{_L}+\sigma_1\left(\frac{\pt f}{\pt x}\right)^*_{_L}\right]-\Delta_6\left[\left(\frac{\pt b}{\pt t}\right)^*_{_L}+\sigma_1\left(\frac{\pt b}{\pt x}\right)^*_{_L}\right]\notag\\
=&~ (g_L)^2\left[\left(\frac{\pt q}{\pt t}\right)_L+\sigma_1\left(\frac{\pt q}{\pt x}\right)_L\right]+\Delta_4\left[\left(\frac{\pt g}{\pt t}\right)_L+\sigma_1\left(\frac{\pt g}{\pt x}\right)_L\right]\notag\\
& -\Delta_2\left[\left(\frac{\pt f}{\pt t}\right)_L+\sigma_1\left(\frac{\pt f}{\pt x}\right)_L\right]-\Delta_3\left[\left(\frac{\pt b}{\pt t}\right)_L+\sigma_1\left(\frac{\pt b}{\pt x}\right)_L\right].\notag
\end{align}
By using \eqref{derivsofU5}-\eqref{derivsofU8}, the derivatives of $x$ can be replaced by ones of $t$. Then we have
\begin{align}
\left[(g^*_L)^2(1-\sigma_1(\vartheta_3)^*_L)+2g^*_L\sigma_1 \Delta_1(\vartheta_2)^*_L\right]&\left(\frac{\pt q}{\pt t}\right)^*_L +\left[\sigma_1 (\vartheta_2)^*_L(g^*_L q^*_L)^2+\frac{2\Delta_1}{g^*_L}(1-\sigma_1 (\vartheta_3)^*_L)\right]\left(\frac{\pt g}{\pt t}\right)^*_L \notag\\
-\left[\left(1-\frac{4\sigma_1}{3u^*_L}\right)\Delta_5+\frac{2\sigma_1}{3(f^*_L)^2}\Delta_6\right]&\left(\frac{\pt f}{\pt t}\right)^*_L-\left[\left(1-\frac{4\sigma_1}{3u^*_L}\right)\Delta_6+\frac{2\sigma_1}{3(b^*_L)^2}\Delta_5\right]\left(\frac{\pt b}{\pt t}\right)^*_L \notag\\
&~~\quad \qquad-\frac{2\sigma_1(\vartheta_1)^*_L}{3u^*_L}[q^*_L(g^*_L)^2+{{2}\Delta_1]\left(\frac{\pt (fb)}{\pt t}\right)^*_L}\notag\\
=(g_L)^2[1-\sigma_1(\vartheta_3)_L+\sigma_1(\vartheta_2)_L\Delta_4]&\left(\frac{\pt q}{\pt t}\right)_L +[(g_Lq_L)^2\sigma_1(\vartheta_2)_L+(1-\sigma_1(\vartheta_3)_L)\Delta_3]\left(\frac{\pt g}{\pt t}\right)_L \notag\\
-\left[\Delta_2(1-\frac{4\sigma_1}{3u_L})+\frac{2\sigma_1}{3(f_L)^2}\Delta_3\right]&\left(\frac{\pt f}{\pt t}\right)_L-\left[\frac{2\sigma_1}{3(b_L)^2}\Delta_2+\Delta_3(1-\frac{4\sigma_1}{3u_L})\right]\left(\frac{\pt b}{\pt t}\right)_L\notag\\
&~~\quad \qquad-\frac{2\sigma_1(\vartheta_1)_L}{3u_L}[q_L(g_L)^2+\Delta_4g_L]\left(\frac{\pt (fb)}{\pt t}\right)_L.\label{eta_shock}
\end{align}
In view of \eqref{xi_shock}, \eqref{tau_shock}, and \eqref{eta_shock}, we obtain the linear system \eqref{Linear_system_Shock}.
\end{proof}
The discussion of 2- and 3-contact waves is similar to the one for the case of Lemma \ref{l-3.2} and thus we directly proceed with the resolution of the 4-rarefaction wave.
\subsection{Resolution of the 4-rarefaction wave}
Here we use the same one-to-one correspondence $(x, t)\rw(\alpha, \beta)$ as shown in the 
1-rarefaction wave.
Thanks to the asymptotic of the GRP to the associated Riemann problem at the singularity
point $(0, 0)^\top$, we denote by $\alpha_R$ the slope of the wave head, by $\alpha$ the speed of the wave speed
inside the rarefaction wave, and by $\alpha^*_R$ the speed of wave tail. Then we have the following
fact.

\begin{prop}\label{prop-singularity_R_2} Consider the curved rarefaction wave associated with the fourth characteristic field $fb+\frac 32 gq$ and denote  $\Theta_R(\alpha):=\frac{\pt t}{\pt \beta}(\al,0)$.  Then we have
\begin{align}
\Theta_R(\alpha)=\Theta_R(\alpha_R)\dfrac{u_R-3v_R}{u_R-3v(\alpha,0)}.\label{thetar}
\end{align}
\end{prop}

\begin{proof}
Looking back to \eqref{t_alphabeta}, we can set $t(\alpha,0)=0$ and $\frac{\pt }{\pt\alpha}(u+\frac 32 v)(\alpha,0)=1$.
If $\lim_{\beta\rw 0^+}\frac{\pt}{\pt\beta}u(\alpha,\beta)$ is bounded, we are led to
\begin{align}
\frac{\pt^2 t}{\pt\alpha\pt\beta}(\alpha,0)=\frac{2}{u-3v}\frac{\pt t}{\pt\beta}(\alpha,0).\notag
\end{align}
The above equation can be written as
\begin{align}
\frac{\pt}{\pt\alpha}\Theta_R(\alpha)=\frac{2}{u-3v}\Theta_R(\alpha),\notag
\end{align}
which leads to
\begin{align}
\frac{\Theta_R(\alpha)}{\Theta_R(\alpha_R)}=\exp\left(\int_{\alpha_R}^{\alpha}\frac{2}{u-3v}~~ d\alpha\right).
\end{align}

The Riemann variants of the 4-rarefaction wave are $f$, $b$ and $\tau=g/q$.
Notice that if $\beta\rw 0$, $\alpha$ can be selected as $\alpha=\frac{x}{t}=u+\frac{3}{2}v$ and $u=fb$ is also constant across the 4-rarefaction wave. This gives 
\begin{align}
    v=\frac 23(\alpha-u_R). 
\end{align}
Thus \eqref{thetar} can be further simplified when $\beta\rw 0$ as follows,
\begin{align*}
\displaystyle\int_{\alpha_R}^{\alpha}\dfrac{2}{u-3v} d\alpha&= \displaystyle\int_{\alpha_R}^{\alpha}\dfrac{2}{3u_R-2\alpha} d\alpha= \ln \left(\dfrac{3u_R-2\alpha_R}{3u_R-2\alpha}\right)
\end{align*}
Then 
\begin{align}
\Theta_R(\alpha)=\Theta_R(\alpha_R)\dfrac{u_R-3v_R}{u_R-3v(\alpha,0)}.\label{thetar1}
\end{align}
\end{proof}
Similar to the 1-rarefaction wave, we now proceed to obtain a linear system for the time derivatives of the conservative variabless across the 4-rarefaction wave. We express this in the following lemma.
\begin{lemma}\label{l-A.2}
    Let the fourth wave be a rarefaction wave. The trace values \[
    \dfrac{\partial\mathbf{U}}{\partial t}(\alpha, 0)=\left[
    \dfrac{\partial f}{\partial t},
    \dfrac{\partial  b}{\partial 
 t},
    \dfrac{\partial  g}{\partial 
 t},
    \dfrac{\partial  q}{\partial  t}\right]^\top(\alpha, 0)
    \]
    satisfy the following linear equations 
\begin{align}\label{Linear_system_R_4}
    \mathbf{A_{R}}(\alpha, 0) \dfrac{\partial \mathbf{U}}{\partial t}(\alpha, 0)= \mathbf{D_{R}}(\alpha, 0), \quad\forall \alpha\in[\alpha^*_R,\alpha_{R}],
\end{align}
where
\begin{align}
    \mathbf{A_{R}}(\alpha, 0)=\begin{bmatrix}
    b(\alpha, 0) & f(\alpha, 0)&0 &0 \\
        \dfrac{1}{b(\alpha, 0)}, & -\dfrac{f}{b^2}(\alpha, 0)& 0 &0\\
        0&0& \dfrac{1}{q}(\alpha, 0)& -\dfrac{g}{q^2}(\alpha, 0)
    \end{bmatrix}, \quad 
\end{align}
and 
\begin{align}
    \mathbf{D_{R}}(\alpha, 0)=\left[
       \left(\dfrac{\partial u}{\partial t}\right)_{R},
    \left(\dfrac{\partial \xi}{\partial t}\right)_{R}, 
   \dfrac{(2u+v)(\alpha, 0)}{(2u_R+v_R)}\left(\dfrac{v(\alpha, 0)}{v_R}\right)^{\frac{1}{2}}\left(\dfrac{\partial \tau}{\partial t}\right)_{R}\right]^\top.
\end{align}
\end{lemma}
\begin{proof}
By using the equation for the Riemann invariant $\tau$ from \eqref {Rtransformed},
it can be rewritten as
\begin{align}
&\frac{\pt\tau}{\pt t}+\frac 32 u\frac{\pt\tau}{\pt x} = \frac 12(u-v)\frac{\pt\tau}{\pt x},\ \frac{\pt\tau}{\pt t}+\left(u+\frac 32 v\right)\frac{\pt\tau}{\pt x} = v\frac{\pt\tau}{\pt x}.\notag
\end{align}
By changing $(x,t)$ to the characteristic coordinates $(\alpha,\beta)$ for $\tau$, it follows
\begin{align}
&\frac{\pt\tau}{\pt\alpha}=\frac{\pt t}{\pt\alpha}\left(\frac{\pt\tau}{\pt t}+\frac 32 u\frac{\pt\tau}{\pt x}\right)=\frac 12(u-v)\frac{\pt t}{\pt\alpha}\frac{\pt\tau}{\pt x},\label{tau1}\\
&\frac{\pt\tau}{\pt\beta}=\frac{\pt t}{\pt\beta}\left(\frac{\pt\tau}{\pt t}+\left(u+\frac 32 v\right)\frac{\pt\tau}{\pt x}\right)=v\frac{\pt t}{\pt\beta}\frac{\pt\tau}{\pt x}.\label{tau2}
\end{align}
We differentiate the first formula with respect to $\beta$, and obtain
\begin{align}
\frac{\pt^2\tau}{\pt\alpha\pt\beta}=&\frac 12(u-v)\frac{\pt^2 t}{\pt\alpha\pt\beta}
\frac{\pt\tau}{\pt x}+\frac 12\frac{\pt t}{\pt \alpha}\frac{\pt}{\pt\beta}\left((u-v)\frac{\pt\tau}{\pt x}\right).\notag
\end{align}
Let $\beta\rw 0^+$ and use the same strategy as for $\Theta_R(\alpha)$ to derive an ODE for $\frac{\pt \tau}{\pt\beta}(\alpha,0)$ given by
\begin{align}
\frac{\pt}{\pt\alpha}\left(\frac{\pt \tau}{\pt\beta}(\alpha,0)\right)&=\frac{u-v}{u-3v}(\alpha,0)\frac{\pt t}{\pt\beta}(\alpha,0)\frac{\pt \tau}{\pt x}(\alpha,0)=\frac{u-v}{v(u-3v)}(\alpha,0)\frac{\pt \tau}{\pt\beta}(\alpha,0).\notag
\end{align}
Then we have
\begin{align}
&\frac{\pt \tau}{\pt\beta}(\alpha,0)=\frac{\pt \tau}{\pt\beta}(\alpha_R,0)\Theta_R(\alpha; \alpha_R)\Upsilon_R^{\tau}(\alpha;\alpha_R),\notag\\
&\Upsilon_R^{\tau}(\alpha;\alpha_R):=\exp\left(\int_{\alpha_R}^{\alpha}\frac{1}{v}d\alpha\right)=\exp\left(\int_{\alpha_R}^{\alpha}\frac{3}{2(\alpha-u_R)}d\alpha\right)=\left(\dfrac{v(\alpha,0)}{v_R}\right)^{\frac 32}.\notag
\end{align}
We go back to the $(x, t)$-coordinates, and  by using \eqref{tau2}, we obtain
\begin{align}
\frac{\pt\tau}{\pt x}(\alpha,0)=\frac{\pt}{\pt x}\left(\frac{g}{q}\right)(\alpha,0)=\left(\dfrac{v(\alpha,0)}{v_R}\right)^{\frac 12}\left(\frac{\pt\tau}{\pt x}\right)_R.
\end{align}
Therefore, we have 
\begin{align}
\left(\frac{\pt \tau}{\pt t}\right)(\alpha, 0)&=-\left(u+\frac 12v\right)(\alpha, 0)\left(\frac{\pt \tau}{\pt x}\right)(\alpha, 0)\notag\\
&=-\left(u+\frac 12v\right)(\alpha, 0)\left(\dfrac{v(\alpha, 0)}{v_R}\right)^{\frac 12}\left(\frac{\pt\tau}{\pt x}\right)_R\notag\\
&=\frac{(2u+v)(\alpha, 0)}{2u_R+v_R}\left(\dfrac{v(\alpha, 0)}{v_R}\right)^{\frac 12}\left(\dfrac{\partial \tau}{\partial t}\right)_R.\label{R4-tau}
\end{align}
The functions $f$ and $b$ are expected to be regular inside the 4-rarefaction wave at the singularity $(0, 0)^\top$. Since $b$ is assumed to be negative, $u=fb$ and $\xi=f/b$ are also regular at the singularity $(0, 0)^\top$.

Consider $\xi$ for the first step. As we know, $\xi$ satisfies the second equation in \eqref{Rtransformed}. It can be rewritten as 
\begin{align}
&\frac{\pt\xi}{\pt t} + \frac 32 u\frac{\pt\xi}{\pt x}=u\frac{\pt\xi}{\pt x},  \frac{\pt\xi}{\pt t} + (u+\frac 32v)\frac{\pt\xi}{\pt x}=\frac 12(u+3v)\frac{\pt\xi}{\pt x}.\notag
\end{align}
By using the relation between $(\alpha,\beta)$ and $(x,t)$, we have
\begin{align}
&\frac{\pt\xi}{\pt \alpha} = u\frac{\pt\xi}{\pt x}\frac{\pt t}{\pt \alpha},\notag\\
&\frac{\pt\xi}{\pt \beta} = \frac 12(u+3v)\frac{\pt\xi}{\pt x}\frac{\pt t}{\pt \beta}.\notag
\end{align}
Differentiating the first equation with respect to $\beta$, we see that
\begin{align}
&\frac{\pt^2\xi}{\pt \alpha\pt\beta} = \frac{\pt}{\pt\beta}\left(u\frac{\pt\xi}{\pt x}\right)\frac{\pt t}{\pt \alpha}+u\frac{\pt\xi}{\pt x}\frac{\pt^2 t}{\pt \alpha\pt\beta}.\notag
\end{align}
Let $\beta\rw 0$ and assume that $\frac{\pt}{\pt\beta}\left(u\frac{\pt\xi}{\pt x}\right)$ is bounded, then we arrive at
\begin{align}
&\frac{\pt^2\xi}{\pt \alpha\pt\beta}(\alpha,0) = \frac{2u}{u-3v}\frac{\pt\xi}{\pt x}\frac{\pt t}{\pt\beta}(\alpha,0)=\frac{4u}{(u-3v)(u+3v)}\frac{\pt \xi}{\pt\beta}(\alpha,0).\notag
\end{align}
Thus $\frac{\pt \xi}{\pt\beta}(\alpha,0)$ can be solved by
\begin{align}
&\frac{\pt \xi}{\pt\beta}(\alpha,0)=\frac{\pt \xi}{\pt\beta}(\alpha_R,0)\Theta_R(\alpha,\alpha_R)\Upsilon_R^{\xi}(\alpha;\alpha_R),\notag\\ 
&\Upsilon_R^{\xi}(\alpha;\alpha_R)=\exp\left(\int_{\alpha_R}^{\alpha}~\frac{2}{u+3v}d\alpha\right)=\exp\left(\int_{\alpha_R}^{\alpha}\frac{2}{2\alpha-u_R}~d\alpha\right)=\frac{u_R+3v(\alpha,0)}{u_R+3v_R}.\notag
\end{align}
In the $(x,t)$-coordinates, we have then
\begin{align}
&\frac{\pt \xi}{\pt x}(\alpha,0)=\left(\frac{\pt \xi}{\pt x}\right)_R.\notag
\end{align}
By using the equation of Riemann invariant $\xi$, we have
\begin{align}
&\frac{\pt \xi}{\pt t}(\alpha,0)=\left(\frac{\pt \xi}{\pt t}\right)_R.\label{R4-xi}
\end{align}
Similar idea can also be applied to $u$ and $\frac{\pt u}{\pt t}(\alpha,0)$ to get
\begin{align}
&\frac{\pt u}{\pt t}(\alpha,0)=\left(\frac{\pt u}{\pt t}\right)_R.\label{R4-u}
\end{align}
Thus, \eqref{R4-tau}, \eqref{R4-xi} and \eqref{R4-u} lead to the linear system \eqref{Linear_system_R_4}.
\end{proof}
\begin{remark}
The vector $\mathbf{D_R}$ is completely known in terms of the initial data of the generalized Riemann problem using \eqref{Rtransformed} similar to Lemma \ref{l-3.3}. Furthermore, the linear system for the fourth rarefaction wave for $\alpha=\alpha^*_R$ coincides with the linear system for the 4-shock wave obtained in Lemma \ref{l-3.3}. This fact resonates with the nature of the fourth characteristic field being a Temple field.  
\end{remark}
\subsection{Time-derivatives at singularity $(0, 0)^\top$ for the wave configuration $S_1+J_2+J_3+R_4$}
Similar to Section \ref{sec: R+J+J+S}, we have 12 unknown time derivatives in the intermediate states, out of which 6 unknown derivatives can be obtained using \eqref{eq: derivative_relations}. 
Then, based on Lemmas \ref{l-A.1}-\ref{l-A.2}, we have the following result for remaining time derivatives.
\begin{theorem}[Instantaneous time-derivatives for wave configuration  $S_1+J_2+J_3+R_4$]\label{non-sonic-case-wave-2}
\rm
The trace values $(\partial \mathbf{U}/\partial t)^*_L, (\partial \mathbf{U}/\partial t)^*_M$ and $(\partial \mathbf{U}/\partial t)^*_R$ in \eqref{GRP-value} are obtained by \eqref{eq: derivative_relations} and by solving an invertible system of linear equations for 
\[\dfrac{\partial\mathbf{U}}{\partial t}=\bigg[
    \left(\dfrac{\partial f}{\partial t}\right)^*_{_L},
   \left(\dfrac{\partial b}{\partial t}\right)^*_{_L},
     \left(\dfrac{\partial g}{\partial t}\right)^*_{_L},
    \left(\dfrac{\partial q}{\partial t}\right)^*_{_L},
    \left(\dfrac{\partial g}{\partial t}\right)^*_{_R},
   \left(\dfrac{\partial q}{\partial t}\right)^*_{_R}\bigg]^{\top}\]
   given by
\begin{align}\label{main_linear_system}
    \mathbf{A(U^*)}\dfrac{\partial \mathbf{U}}{\partial t}= \mathbf{D(U_{L}, U_{R},U^*,U_{L}^{\prime}, U_{R}^{\prime})}.
\end{align}
The coefficient matrix $\mathbf{A}$ and the vector $\mathbf{D}$ only depend on the initial data or the intermediate states of the associated Riemann problem \eqref{RP-data} and are given by
\begin{align}
    \mathbf{A(U^*)}=\begin{bmatrix}
        b^*_L& f^*_L&0&0&0&0\\
        \dfrac{1}{b^*_L}&-\dfrac{f^*_L}{(b^*_L)^2}&0 &0&0&0\\
        \mathcal{A}^*_L & \mathcal{B}^*_L & \mathcal{C}^*_L& \mathcal{D}^*_L&0&0\\
        0 & 0 & \dfrac{1}{q^*_L}&-\dfrac{g^*_L}{(q^*_L)^2} &0&0\\
        0 & 0 &-q^*_L& -g^*_L& q^*_R& g^*_R\\
        0 & 0 & 0 & 0&\dfrac{1}{q^*_R}&-\dfrac{g^*_R}{(q^*_R)^2}\\
    \end{bmatrix},
\end{align}
 and
 \begin{align}
\mathbf{D(U_{L}, U_{R},U^*,U_{L}^{\prime}, U_{R}^{\prime})}=\begin{bmatrix}
    \left(\dfrac{\partial u}{\partial t}\right)_{R}\vspace{0.2 cm}\\
    \left(\dfrac{f_Rb_R}{f_Lb_L}\right)^{3/2}\left(\dfrac{\partial \xi}{\partial t}\right)_{L}\vspace{0.2 cm}\\
    \mathcal{A}_L\left(\dfrac{\partial f}{\partial t}\right)_L+\mathcal{B}_L\left(\dfrac{\partial b}{\partial t}\right)_L+\mathcal{C}_L\left(\dfrac{\partial g}{\partial t}\right)_L+\mathcal{D}_L\left(\dfrac{\partial q}{\partial t}\right)_L\vspace{0.2 cm}\\    \dfrac{(2u^*_L+v^*_L)(2u_L+v_L-2\sigma_1)}{(2u_L+v_L)(2u^*_L+v^*_L-2\sigma_1)}\left(\dfrac{\partial \tau}{\partial t}\right)_{L} \vspace{0.2 cm} \\
    0\vspace{0.2 cm}\\
\bigg(\dfrac{2u^*_R+v^*_R}{2u_R+v_R}\sqrt{\dfrac{v^*_R}{v_{R}}}\bigg)\bigg(\dfrac{\partial \tau}{\partial t}\bigg)_{R}
\end{bmatrix}, 
\end{align}
where the coefficients $\mathcal{A}^*_L, \mathcal{B}^*_L, \mathcal{C}^*_L, \mathcal{D}^*_L, \mathcal{A}_{L}, \mathcal{B}_{L}, \mathcal{C}_{L}$ and $\mathcal{D}_{L}$ are defined as in Lemma \ref{l-A.2}.
\end{theorem}
The construction of the acoustic case for this wave configuration is similar to that discussed in the Section \ref{sec: acoustic}. Therefore, for the sake of brevity, we do not discuss that here.

\section{Numerical tests}\label{sec: 5}
In this section, the accuracy of the GRP-based finite-volume method is investigated numerically. To display the advantage of the GRP method over other second-order methods, we also compare the results obtained by the GRP method with those for the Godunov method and the MUSCL method with the second-order RK (RK-2) time stepping. The exact Riemann solver is used in the MUSCL method. 

In the simulations, we use the same time step conditions as for the Godunov scheme and denote the number of cells by $N$. 
For the test cases with the smooth solution, we update
the slope by $\bm{\sigma}_j^{k+1}=(\mathbf{U}_{j+1/2}^{k+1, -}-\mathbf{U}_{j-1/2}^{k+1, -})/\Delta x
$ in \eqref{piecewise linear initial data}, where $\mathbf{U}_{j+1/2}^{k+1, -}$ is defined in \eqref{updated_slope_value}. 
Otherwise, in order to suppress local oscillations near discontinuities, the parameter $\theta$ in \eqref{slope_update} is taken from $[0, 2)$. In practice, we use large values of $\theta$ for less diffusive results and smaller values of $\theta$ near really sharp discontinuities to mitigate the oscillations.
\begin{example}{(Smooth travelling wave solutions for the system \eqref{eq: Main_system})}
\rm

In this example, a particular travelling wave solution for system \eqref{eq: Main_system} is considered and compared with the numerical solutions obtained from our proposed solver to check the accuracy. In particular, we look for a solution of the form
\begin{align}
    f(x, t)=F(\chi), \quad b(x, t)=B(\chi), \quad g(x, t)=G(\chi), \quad q(x, t)=Q(\chi),
\end{align}
where $\chi =x-ct$ for some $c\in \mathbb{R}$.

Then, system \eqref{eq: Main_system} can be reduced to  a system of ODEs given by
\begin{align*}
    -c F'(\chi)+\dfrac{d}{d\chi}\left(\dfrac{F^2 B}{2}\right)&=0,\\
    -c B'(\chi)+\dfrac{d}{d\chi}\left(\dfrac{F B^2}{2}\right)&=0,\\
    -c G'(\chi)+\dfrac{d}{d\chi}\left(\dfrac{G^2 Q}{2}+FBG\right)&=0,\\
    -c Q'(\chi)+\dfrac{d}{d\chi}\left(\dfrac{G Q^2}{2}+FBQ\right)&=0.
\end{align*}
Integrating these equations, we obtain
\begin{align*}
    F^2B-2cF&=A_1,\\
    FB^2-2cB&=A_2,\\
    G^2Q+2FBG-2cG&=A_3,\\
    GQ^2+2FBQ-2cQ&=A_4.
\end{align*}
Let us choose $A_1=A_2=A_3=A_4=0$, which implies that nontrivial travelling wave solutions of system \eqref{eq: Main_system} must satisfy
\begin{align}
    FB=2c,\quad GQ=-2c.
\end{align}
Note that in this case, all the characteristic fields of the system \eqref{eq: Main_system} become linearly degenerate and the existence of travelling wave solutions can be ensured (see e.g. \cite{liu2013existence}). Such a test case has been utilized for the test of accuracy of well-known methods for compressible Euler equations; see e.g. \cite{li2016high}.
\begin{table}[ht]
    \centering

\resizebox{\linewidth}{!}{
\begin{tabular}{|c|c|c|c|c|c|c|c|c|c|c|c|c|}
\hline
$ $ & \multicolumn{6}{c|}{\text{GRP}} & \multicolumn{6}{c|}{\text{MUSCL}} \\
\hline
$N$ & $L^1$\text{-error} & \text{Order} & $L^2$\text{-error} & \text{Order} & $L^\infty$\text{-error} & \text{Order} & $L^1$\text{-error} & \text{Order} & $L^2$\text{-error} & \text{Order} & $L^\infty$\text{-error} & \text{Order} \\
\hline
$20$ & $1.39\times10^{-1}$ & $-$ & $6.08\times10^{-2}$ & $-$ & $3.88\times10^{-2}$ & $-$ & $3.15\times10^{-1}$ & $-$ & $1.59\times10^{-1}$ & $-$ & $1.28\times10^{-1}$ & $-$ \\
$40$ & $3.50\times10^{-2}$ & $1.99$ & $1.54\times10^{-2}$ & $1.98$ & $1.04\times10^{-2}$ & $1.90$ & $8.10\times10^{-2}$ & $1.96$ & $4.17\times10^{-2}$ & $1.93$ & $3.37\times10^{-2}$ & $1.92$ \\
$80$ & $8.75\times10^{-3}$ & $2.00$ & $3.89\times10^{-3}$ & $1.99$ & $2.70\times10^{-3}$ & $1.94$ & $2.06\times10^{-2}$ & $1.97$ & $1.07\times10^{-2}$ & $1.96$ & $8.44\times10^{-3}$ & $2.00$ \\
$160$ & $2.19\times10^{-3}$ & $2.00$ & $9.76\times10^{-4}$ & $2.00$ & $6.89\times10^{-4}$ & $1.97$ & $5.18\times10^{-3}$ & $2.00$ & $2.70\times10^{-3}$ & $1.99$ & $2.24\times10^{-3}$ & $1.92$ \\
$320$ & $5.46\times10^{-4}$ & $2.00$ & $2.44\times10^{-4}$ & $2.00$ & $1.74\times10^{-4}$ & $1.99$ & $1.29\times10^{-3}$ & $2.00$ & $6.75\times10^{-4}$ & $2.00$ & $5.71\times10^{-4}$ & $1.97$ \\
$640$ & $1.37\times10^{-4}$ & $2.00$ & $6.11\times10^{-5}$ & $2.00$ & $4.36\times10^{-5}$ & $1.99$ & $3.24\times10^{-4}$ & $2.00$ & $1.69\times10^{-4}$ & $2.00$ & $1.44\times10^{-4}$ & $1.99$ \\
\hline
\end{tabular}}
  \caption{The errors and orders of convergence of the film height $f$ at $t = 3.0$ obtained by the GRP method and the MUSCL method with RK-2 time-stepping using
   $\rm{CFL}= 0.4$ for the travelling wave solution \eqref{eq: travelling_wave}.}
  \label{tab: nolimiters}
\end{table}

\begin{table}[ht]
  \centering
 \resizebox{\linewidth}{!}{
\begin{tabular}{|c|c|c|c|c|c|c|c|c|c|c|c|c|}
\hline
$ $ & \multicolumn{6}{c|}{\text{GRP}} & \multicolumn{6}{c|}{\text{MUSCL}} \\
\hline
$N$ & $L^1$\text{-error} & \text{Order} & $L^2$\text{-error} & \text{Order} & $L^\infty$\text{-error} & \text{Order} & $L^1$\text{-error} & \text{Order} & $L^2$\text{-error} & \text{Order} & $L^\infty$\text{-error} & \text{Order} \\
\hline
$20$ & $1.14\times10^{-1}$ & $-$ & $6.40\times10^{-2}$ & $-$ & $5.55\times10^{-2}$ & $-$ & $3.27\times10^{-1}$ & $-$ & $1.82\times10^{-1}$ & $-$ & $1.51\times10^{-1}$ & $-$ \\
$40$ & $3.02\times10^{-2}$ & $1.92$ & $1.69\times10^{-2}$ & $1.93$ & $1.59\times10^{-2}$ & $1.80$ & $8.79\times10^{-2}$ & $1.90$ & $5.38\times10^{-2}$ & $1.76$ & $4.66\times10^{-2}$ & $1.70$ \\
$80$ & $7.63\times10^{-3}$ & $1.99$ & $4.26\times10^{-3}$ & $1.99$ & $4.19\times10^{-3}$ & $1.93$ & $2.18\times10^{-2}$ & $2.01$ & $1.36\times10^{-2}$ & $1.98$ & $1.24\times10^{-2}$ & $1.91$ \\
$160$ & $1.91\times10^{-3}$ & $2.00$ & $1.07\times10^{-3}$ & $2.00$ & $1.06\times10^{-3}$ & $1.99$ & $5.34\times10^{-3}$ & $2.03$ & $3.35\times10^{-3}$ & $2.02$ & $3.10\times10^{-3}$ & $2.00$ \\
$320$ & $4.77\times10^{-4}$ & $2.00$ & $2.66\times10^{-4}$ & $2.00$ & $2.65\times10^{-4}$ & $2.00$ & $1.32\times10^{-3}$ & $2.01$ & $8.30\times10^{-4}$ & $2.01$ & $7.72\times10^{-4}$ & $2.01$ \\
$640$ & $1.19\times10^{-4}$ & $2.00$ & $6.66\times10^{-5}$ & $2.00$ & $6.63\times10^{-5}$ & $2.00$ & $3.29\times10^{-4}$ & $2.01$ & $2.07\times10^{-4}$ & $2.01$ & $1.92\times10^{-4}$ & $2.00$ \\
\hline
\end{tabular}}
\caption{The errors and orders of convergence of the concentration gradient $b$ at $t = 3.0$ obtained by the GRP method and the MUSCL method with RK-2 time-stepping using
   $\rm{CFL}= 0.4$ for the travelling wave solution \eqref{eq: travelling_wave}.}
  \label{tab: error_b}
\end{table}

We set $c=-1$ and consider the following $2\pi$-periodic solution of the system \eqref{eq: Main_system} 
\begin{align}\label{eq: travelling_wave}
    f(x, t)&= 2+\sin(x+t),\, b(x, t)=-\dfrac{2}{f(x, t)}, \, g(x, t)= 2, \, q(x, t)= 1.
\end{align}
Notably, this type of solution also corresponds to the solution of the one-layer thin film flow model, as $f$ and $b$ are independent of the evolution of $g$ and $q$. Physically, such a solution represents the case where the Marangoni effect in the bottom layer dominates the flow towards the left.

To investigate the accuracy of our proposed solver and the convergence rates, we compute $L^1$-, $L^2$- and $L^{\infty}$-errors in the film height $f$ and the concentration gradient $b$ at time $t=3.0$ with $\rm{CFL}=0.4$ and present them in  Tables \ref{tab: nolimiters}, \ref{tab: error_b}. We also plot the convergence rates for this case in \ref{fig:convergence_f}-\ref{fig:convergence_b}. It can be observed from Tables \ref{tab: nolimiters}, \ref{tab: error_b} that the orders of convergence for both variables are almost $2$ for this test case. Moreover, the numerical error of the GRP method is much less than that of the MUSCL method.
\begin{figure}
    \centering
    \begin{subfigure}[b]{0.45\linewidth}
  \includegraphics[width={\linewidth}]{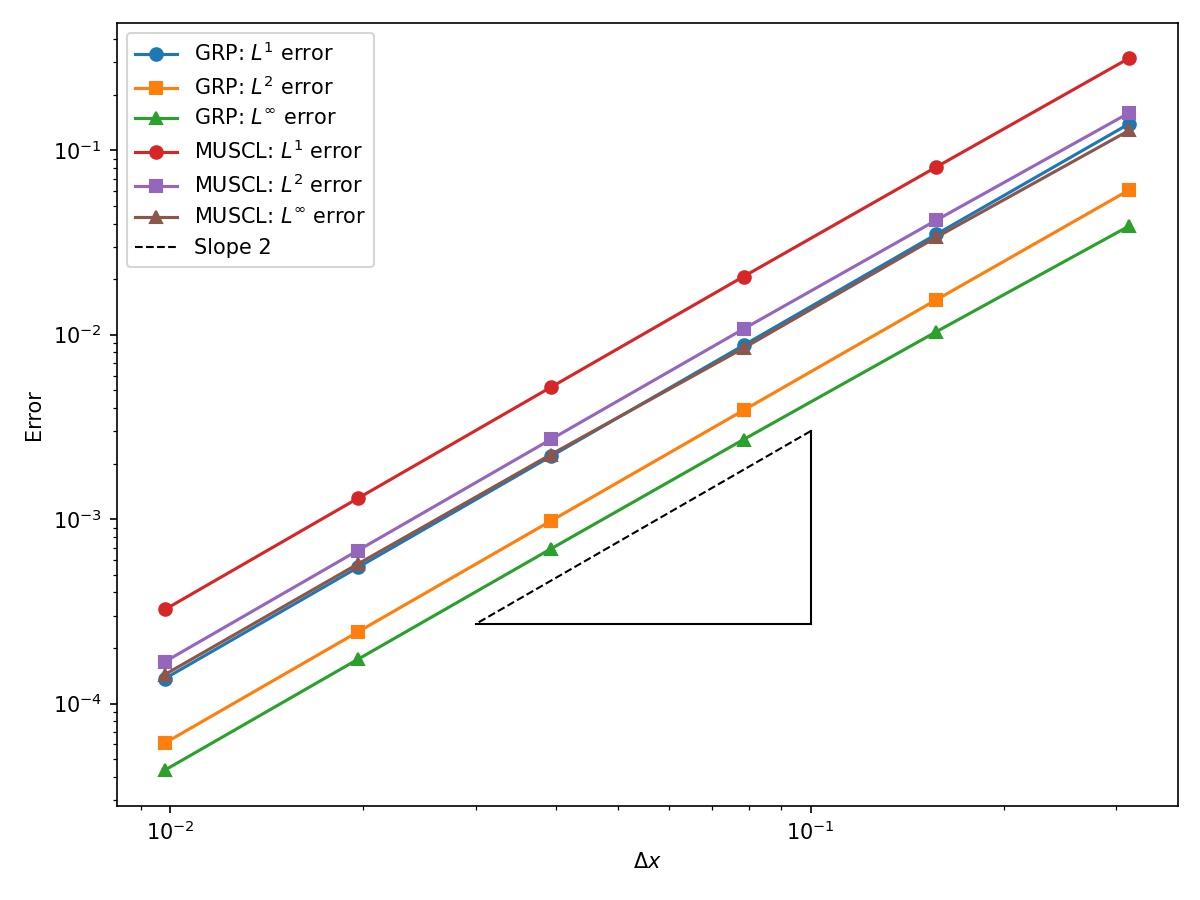}
    \caption{Convergence rates for film thickness $f$.}
    \label{fig:convergence_f}
    \end{subfigure}
    \hfill%
    \begin{subfigure}[b]{0.45\linewidth}
  \includegraphics[width={\linewidth}]{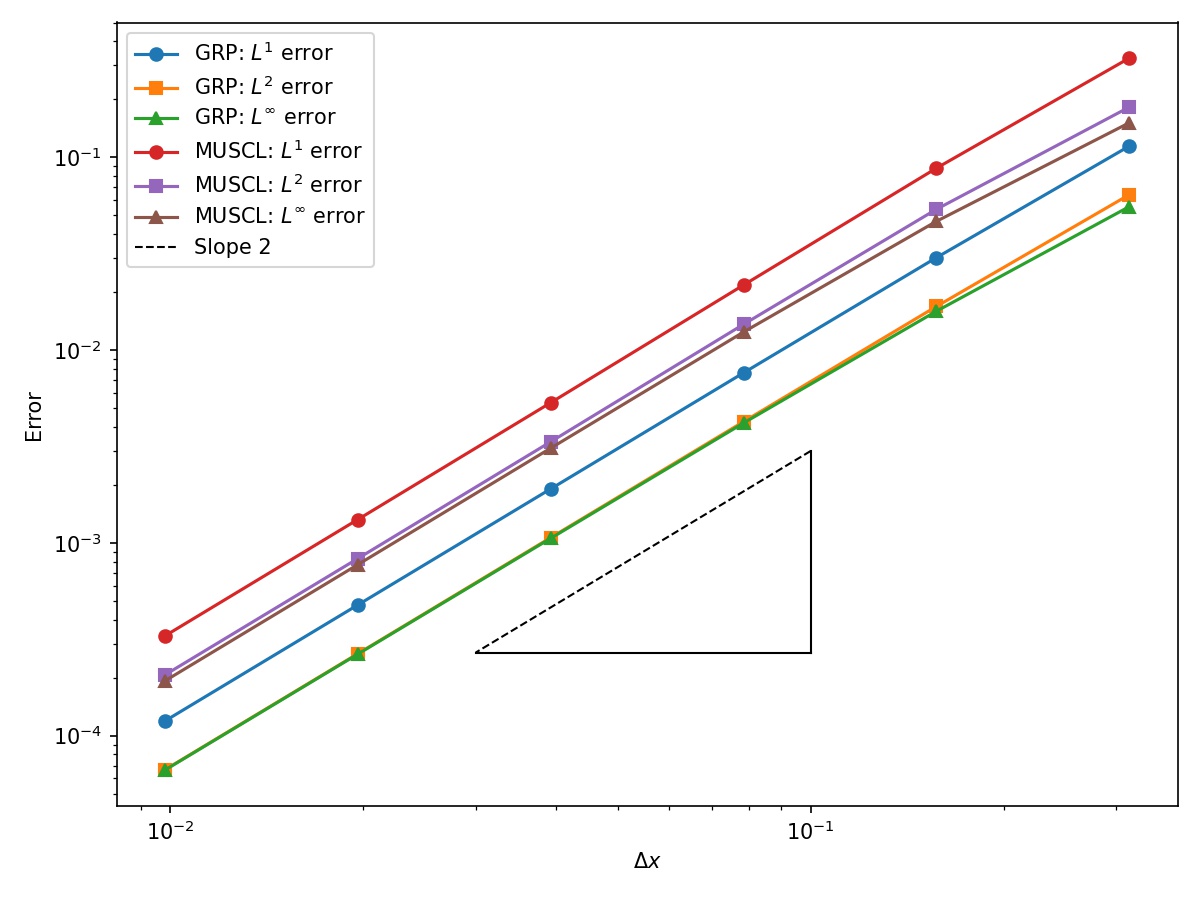}
    \caption{Convergence rates for concentration gradient $b$.}
    \label{fig:convergence_b}
    \end{subfigure}
\end{figure}

\end{example}

\begin{example}{(Pure rarefaction)}\label{pure_rarefaction}
\rm
\begin{figure}
    \centering
    \includegraphics[width=\linewidth]{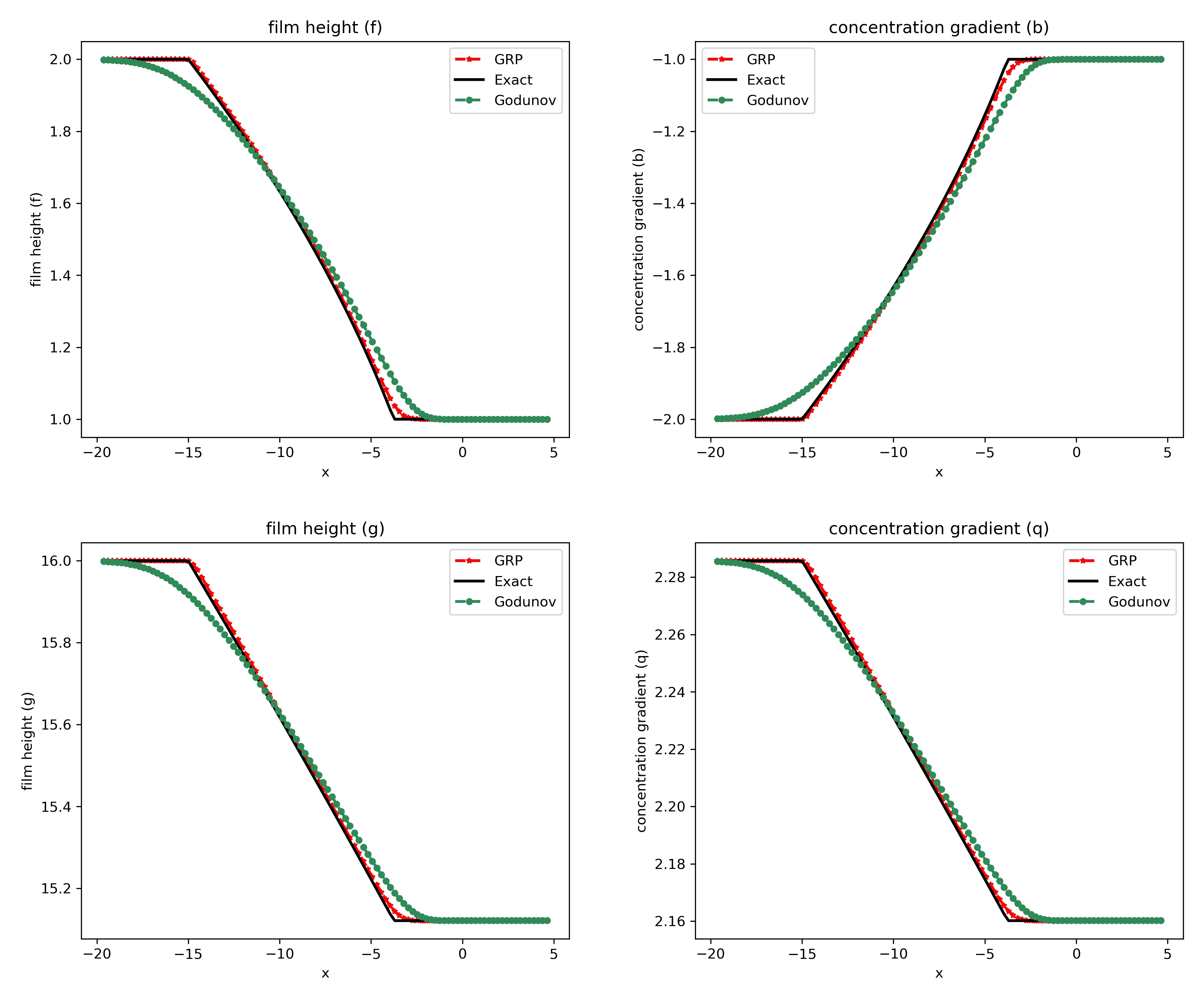}
    \caption{Comparison of the GRP method with the Godunov method for Example \ref{pure_rarefaction} at $t=2.5$ with $\rm{CFL}=0.4$ and $N=100 ~(\Delta x=0.25)$.}
    \label{fig:rarefaction}
\end{figure}
This case is considered to test the pure rarefaction wave, which is similar to \cite{munkejord2006multi}. We consider the Riemann data initially
\begin{align}
    \mathbf{U}(x, 0)=\begin{cases}
        (2.0, -2.0, 16.00, 2.286), \qquad -20\leq x<0,\\
        (1.00, -1.00, 4.00, 0.57143), \qquad 0<x\leq 5.
    \end{cases}
\end{align}
For this case, all the elementary waves collapse into one single 1-rarefaction wave. The computational domain is $[-20, 5]$. In Figure \ref{fig:rarefaction}, the numerical solution obtained by the GRP method at time $t=2.5$ with $N=100$ cells are plotted. It is compared with the results obtained by the Godunov method. It can be observed that the resolution of the numerical solution is much better with the GRP method.
\end{example}
\begin{example}{(The Riemann problem)}\label{example:7.3}
\rm
In this case, we test the GRP method with the Riemann initial data
\begin{align}
    \mathbf{U}(x, 0)=\begin{cases}
        ( 1.57, -1.15, 2.5, 1.90), \quad -10\leq x\leq 10\\
        (1.9, -0.58, 2.4, 2.30), \qquad 10<x\leq 40.
    \end{cases}
\end{align}
The final time is set to be $t=3.5$. The solution to the Riemann problem in this case consists of two rarefaction waves separated by two contact discontinuities. We use $200$ cells in the simulations and compare the results obtained by the GRP method with the MUSCL-RK2 method. We plot the solutions for this case in Figure \ref{fig:R+R}. One can see that the GRP method performs much better in the resolution of both the continuous waves and the discontinuity, and that the MUSCL scheme is found to be more diffusive than the GRP method.

\begin{figure}
    \centering
    \includegraphics[width=\linewidth]{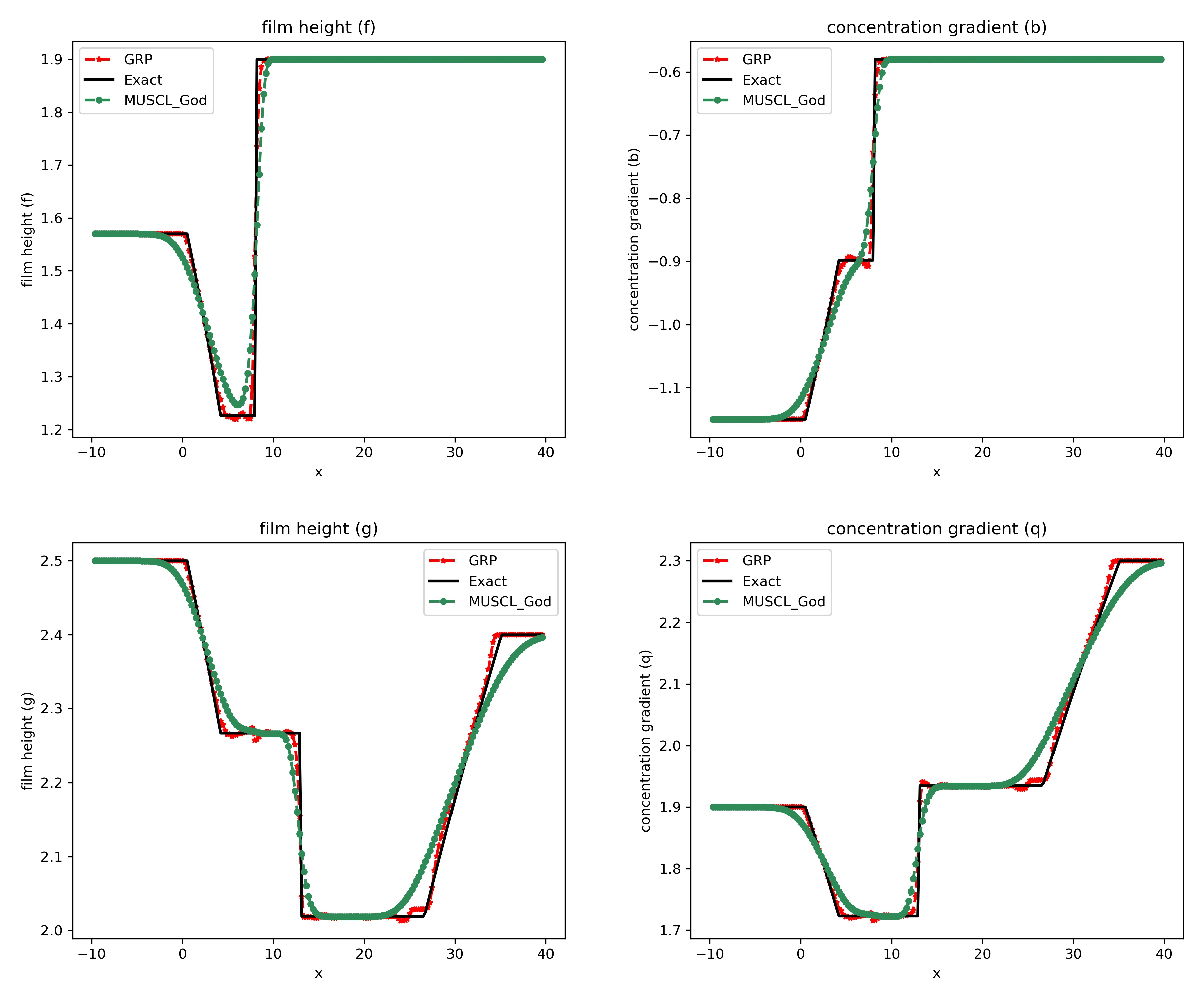}
    \caption{Comparison of the GRP method and the MUSCL-RK2 scheme for Example \ref{example:7.3}} at $t=3.50$ with $\rm{CFL}=0.4$ and $N=200~(\Delta x=0.25)$.
    \label{fig:R+R}
\end{figure}
\end{example}
\begin{figure}
    \centering
    \begin{subfigure}[b]{0.45\linewidth}
        \includegraphics[width={\linewidth}]{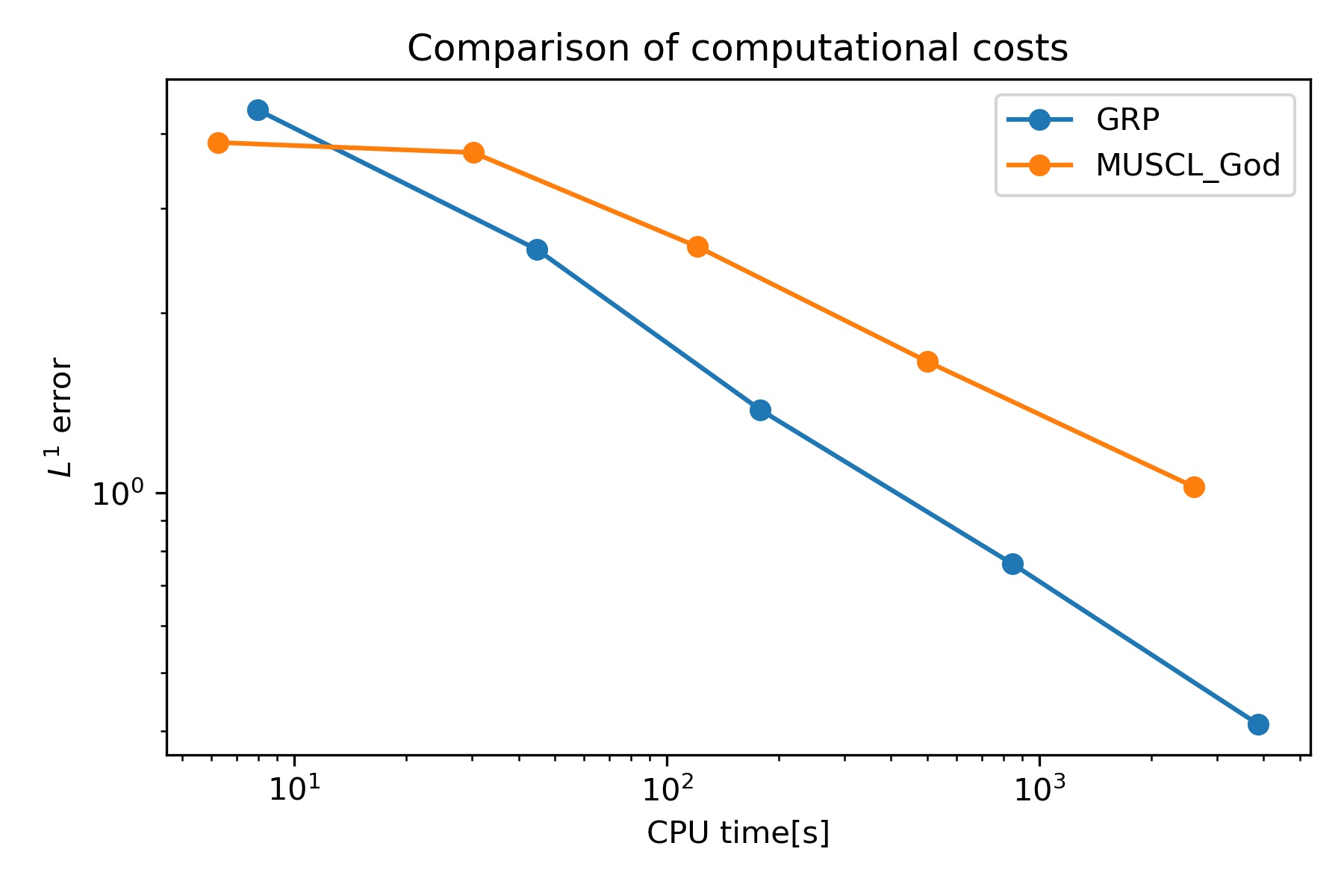}
    \caption{Comparison of computational times for Example \ref{large_ratio}.}
    \label{CPU_Example5.4}
    \end{subfigure}
    \hfill%
    \begin{subfigure}[b]{0.45\linewidth}
        \includegraphics[width={\linewidth}]{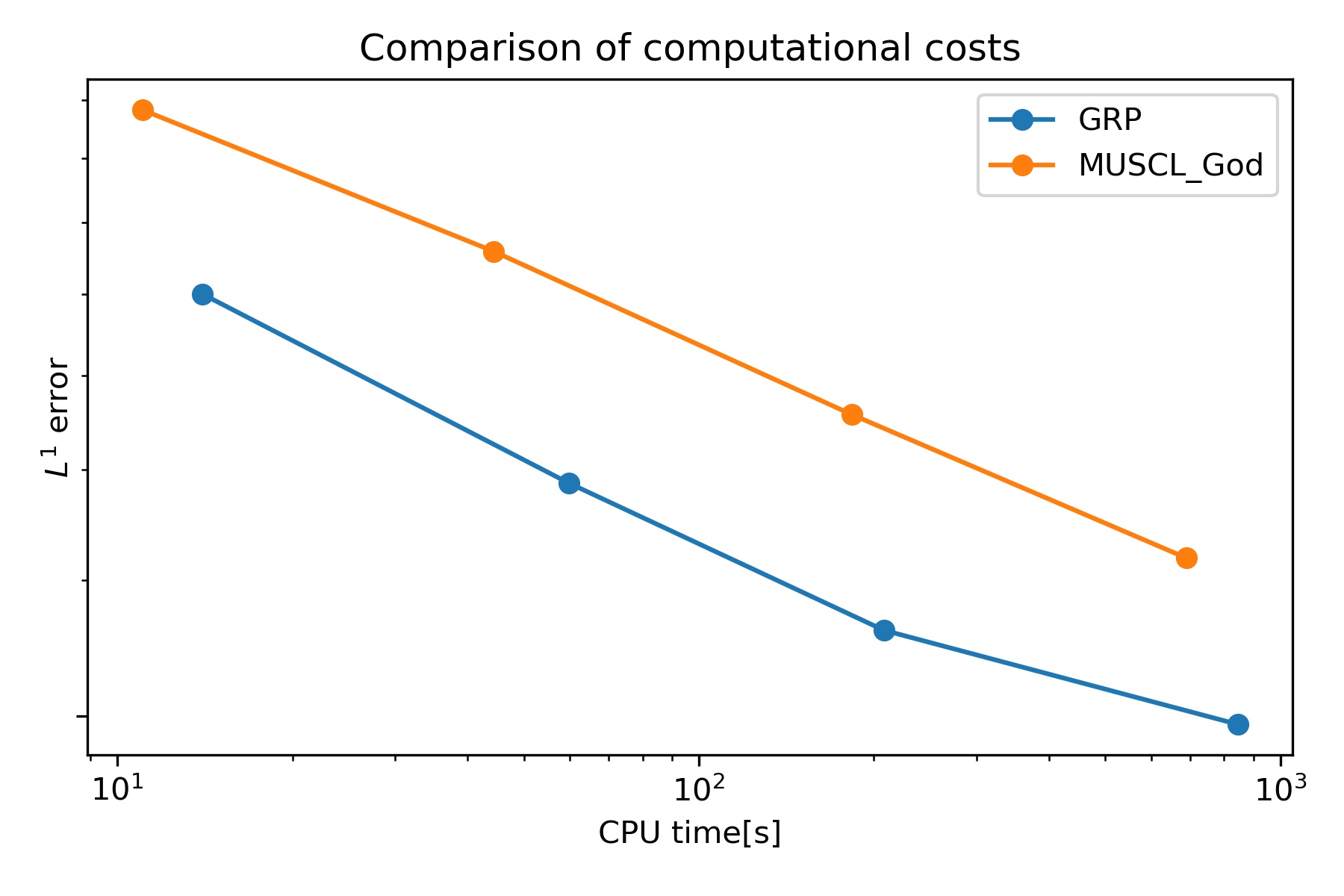}
    \caption{Comparison of computational times for Example \ref{shock-tube}.}
    \label{CPU_Example5.5}
    \end{subfigure}
    
\end{figure}
\begin{example}{(Large height ratio problem)}\label{large_ratio}
Similar to the large density ratio problem for compressible Euler equations \cite{tang2006note}, here we consider initial data where the ratio of film height is quite large. This test case is used to test the ability of the method to capture a strong 1-rarefaction wave as well as the 4-shock wave location. Here we consider the Riemann data:
\begin{align}
    \mathbf{U}(x, 0)=\begin{cases}
        (1.0, -1.5, 2.2, 1.3), \quad -15\leq x<0,\\
        (0.125, -1.5, 0.9, 0.9), \quad 0<x\leq 10.
    \end{cases}
\end{align}
In this case, the computational domain is $[-15, 10]$ and the final time is $t=5.0$. Due to the large jump in the film heights, a strong rarefaction is produced, followed by discontinuities. We plot the obtained results in Figure \ref{fig:large} with $100$ cells and compare them with the solutions of the GRP method with the MUSCL-RK-2 scheme. It can be observed from the numerical results that the GRP method captures the strong rarefaction wave and the shock location better than the MUSCL-RK2 scheme. This shows that GRP has less numerical dissipation and a remarkable ability to capture the locations of discontinuities. To display the computational advantage of the GRP scheme over the MUSCL scheme with RK-2 time stepping, we plot the $L^1$-error vs CPU time in \ref{CPU_Example5.4}. One can observe that to attain the same level of accuracy, GRP takes much less computational time as compared to the MUSCL scheme.
\begin{figure}
    \centering
\includegraphics[width=\linewidth]{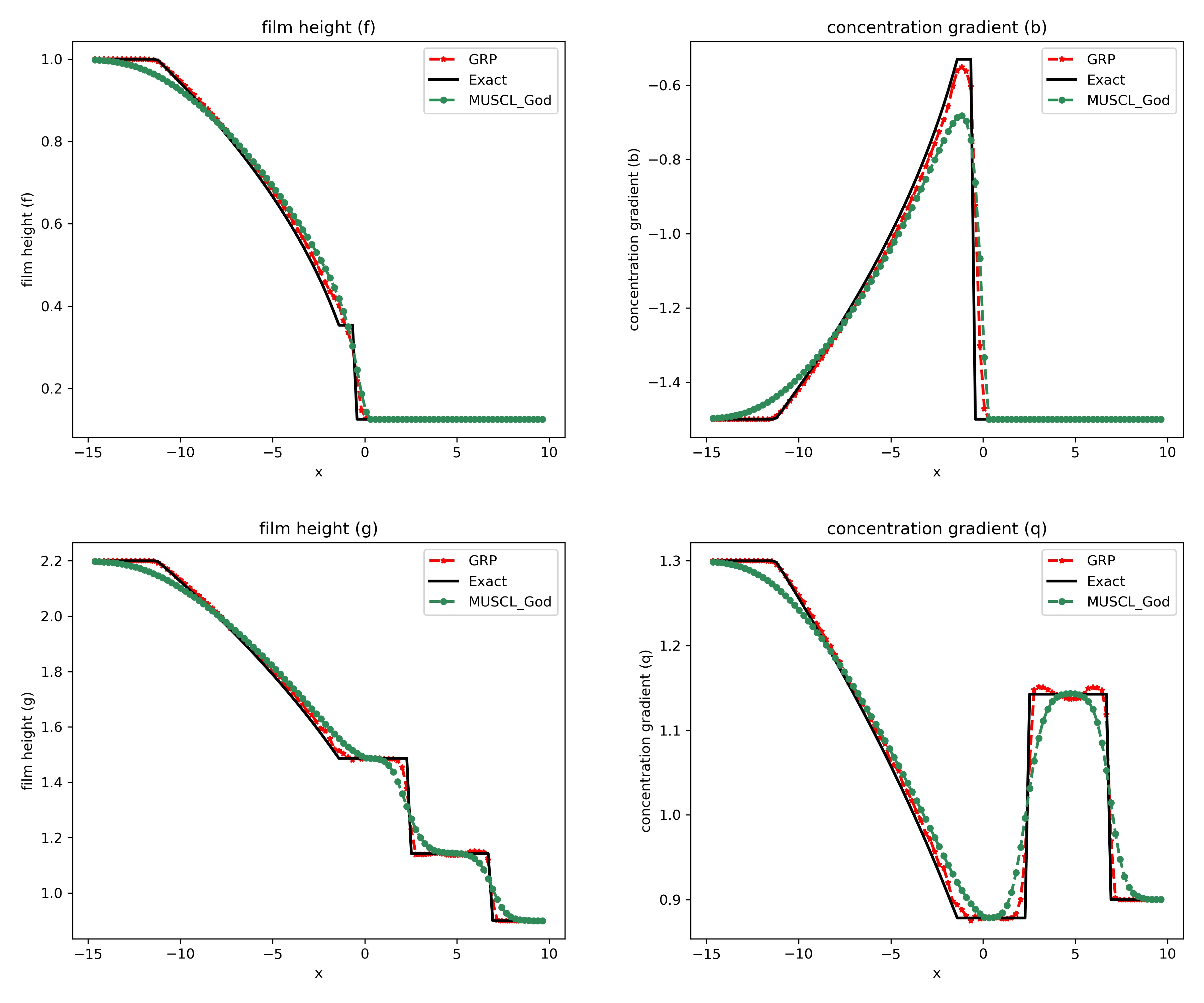}
    \caption{Comparison of the GRP method with the MUSCL-RK2 scheme for Example \ref{large_ratio} with $N=100 (\Delta x=0.25)$ and $\rm{CFL}= 0.4$ at time $t=5.0$.}
    \label{fig:large}
\end{figure}
\end{example}
\begin{example}{(Shock tube type problem)}\label{shock-tube}
We check the performance of the GRP method to capture discontinuities. For this case, we choose initial data of the form 
\begin{align}
    \mathbf{U}(x, 0)=\begin{cases}
        (1.57, -0.95 , 3.1, 1.50), \quad -10\leq x<0,\\
        (1.45, -1.18, 3.6, 1.10), \quad 0<x\leq 15.
    \end{cases}
\end{align}
The final time of simulation is $t=2.5$ and the numerical results are plotted in Figure \ref{fig:shock} using $N=100$ cells. It can be observed again that GRP captures the discontinuities much better than the MUSCL scheme, which validates the accuracy of GRP compared to its second-order counterpart. We also plot the $L^1$-error vs CPU time in Figure \ref{CPU_Example5.5}, which displays the efficiency of GRP scheme once again.
\begin{figure}
    \centering
\includegraphics[width=\linewidth]{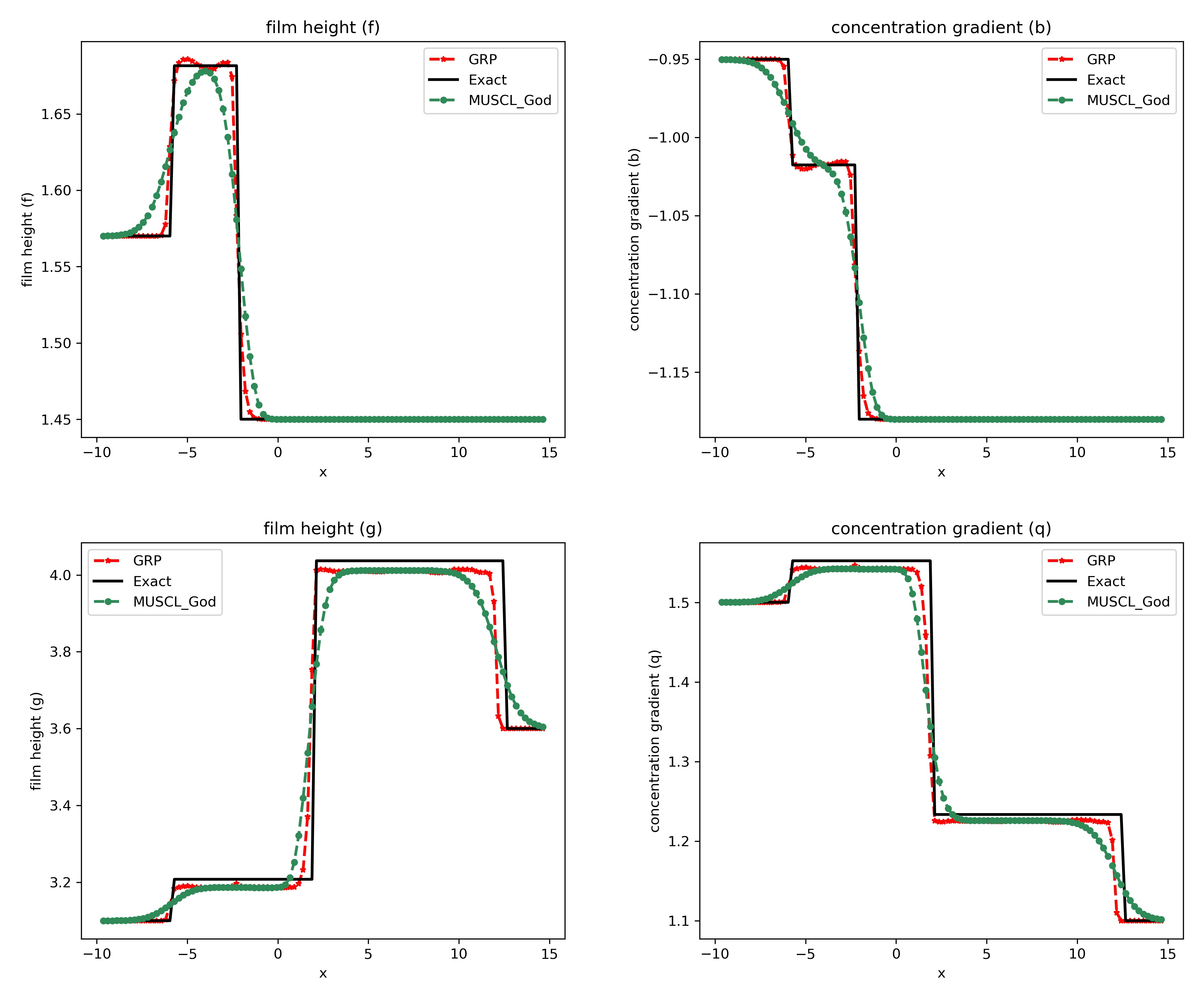}
    \caption{Comparison of the GRP method with the second-order MUSCL scheme with RK-2 time stepping for the shock-tube problem at $t=2.5$ with $N=100 (\Delta x=0.25)$ and $\rm{CFL}=0.4$.}
    \label{fig:shock}
\end{figure}
\end{example}
\begin{example}[Effect of concentration gradients on film heights]
To understand the Marangoni effect and the evolution of film heights due to the presence of solute particles, we assume that the film heights are initially kept at a constant level, while the concentration gradients in each layer follow a Gaussian profile, i.e. we consider the initial data
\begin{align}
    f(x, 0)=1, \, b(x, 0)=-1-\exp(-(x-4)^2), \, g(x, 0)=2, \, q(x, 0)=2+\exp(-(x-4)^2).
\end{align}
The computational domain is chosen to be $[-20, 40]$ and the times of simulations are $t=1.0, 2.0, 3.0, 4.0$ and $5.0$, respectively. For this choice of initial data, we don't have an exact solution available. We plot the solution profile using the proposed scheme with $N=400$ and $\rm{CFL}=0.4$ in Figure \ref{fig:snapshots}. It can be observed from Figure \ref{fig:snapshots} that discontinuities evolve in both film heights. In the first layer, the discontinuities move towards the left due to the Marangoni effect of interfacial stress and remain unaffected by the surface tension variation, which resonates with the evolution of $f$ being independent of $g$ and $q$. However, the discontinuities in the film height of the second layer split into two major discontinuities, out of which one moves to the left and one moves to the right due to the competing Marangoni effects in the second layer.
\begin{figure}
    \centering
    \includegraphics[width=\linewidth]{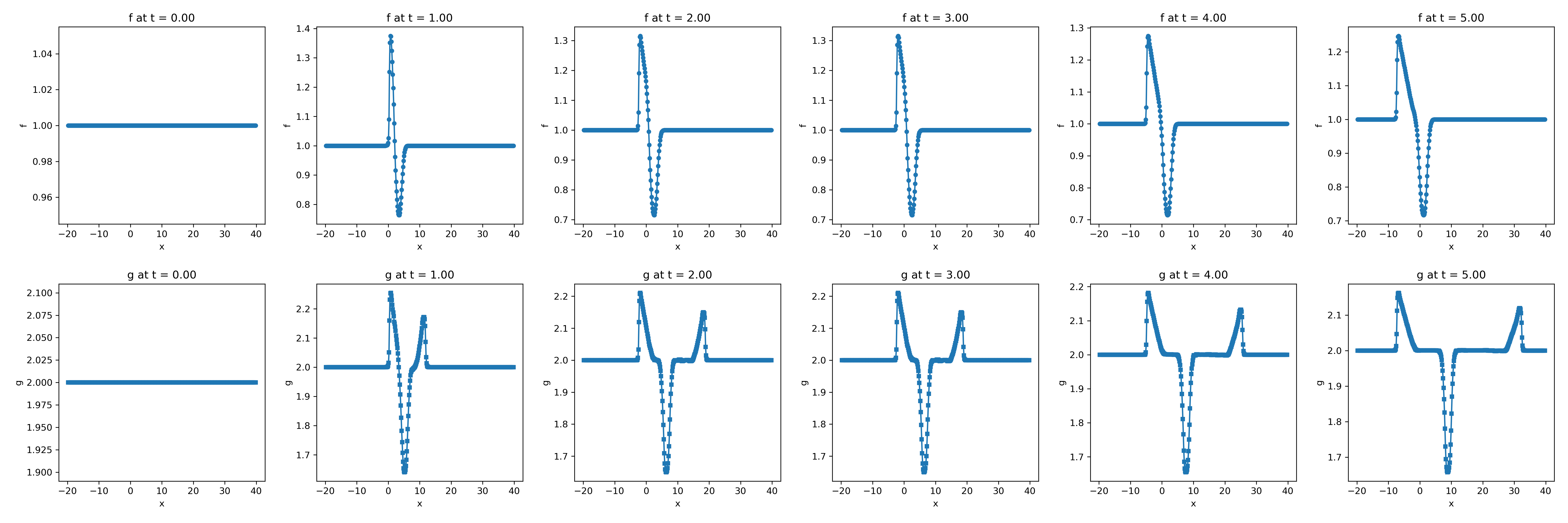}
    \caption{Evolution of film heights  $f$ (top line) and $g$  for a Gaussian concentration gradient profile.}
    \label{fig:snapshots}
\end{figure}
\end{example}
\section{Conclusions and future scope}\label{sec: 6}
We developed a second-order GRP solver for the hyperbolic system of conservation laws \eqref{eq: Main_system}, which governs two-layer thin film flows. 
The explicit time derivatives of the conservative variables are derived by solving a system of linear equations for each possible wave configuration in the generalized Riemann problem so that the resulting scheme is temporal-spatial coupling second order. 
Numerical simulations indicate that the GRP method performs better than, e.g.,  the MUSCL scheme with RK-2 time stepping, in capturing continuous and discontinuous solutions of the hyperbolic system \eqref{eq: Main_system}. 
Since the GRP method provides a set of linear equations to compute time derivatives, the computational cost is comparable to that of other approaches.

Extending GRP solvers to obtain a stable second-order or higher-order entropy-stable discontinuous Galerkin scheme is a natural next step. Moreover, the use cases of the constructed GRP solver are not limited for the hyperbolic system of conservation law \eqref{eq: Main_system}, but can also be utilized as a tool for the numerical simulation for higher-order evolution equations of thin film flows in one or higher space dimensions.
\appendix
\section{Appendix A: Other possible wave configurations with different state spaces}\label{appendix-B}
In the main part of the paper, we focus on the state space \eqref{statespace}. However, one can develop the GRP solver for \eqref{eq: Main_system} in other state spaces using the same methodology as discussed in Section \ref{sec: 3} and Section \ref{sec: 4}. For the sake of clarity, we discuss other possible state spaces and the corresponding wave configurations in this appendix.
\begin{enumerate}
\item[(i)] Consider the state space
\begin{align}\label{statespace_1}
    \mathcal{U}_1=\{(f, b, g, q)\in \mathbb{R}^4: f, g, q>0, b<0, fb+gq< 0, fb+3gq>0\}.
\end{align}
In this state space, the eigenvalues follow the ordering
  \begin{align}
    \frac 32fb<fb+\frac 12gq<\frac 12fb<fb+\frac 32gq.
\end{align}
Therefore, the wave configuration is (R/S)-J-J-(R/S), but the two contact waves change their roles. In this state space, the Sonic case $(fb+3gq/2=0)$ is possible. However, we don't discuss that here for the sake of brevity.
\item[(ii)] Consider the state space 
 \begin{align}\label{statespace_2}
    \mathcal{U}_2=\{(f, b, g, q)\in \mathbb{R}^4: f, g, q>0, b<0, fb+3gq<0\}.
\end{align}
In this state space, the eigenstructure is
  \begin{align}
    \frac 32fb<fb+\frac 12gq<fb+\frac 32gq<\frac 12fb.
\end{align}
Therefore, the wave configuration is (R/S)-J-(R/S)-J, and the construction of the GRP solver is similar. 
\item[(iii)] Consider the state space
 \begin{align}\label{statespace_3}
    \mathcal{U}_3=\{(f, b, g, q)\in \mathbb{R}^4:  f, g, q, b>0, fb<gq.\}.
\end{align}
The eigenstructure for this case is  \begin{align}
    \frac 12fb<\frac 32fb<fb+\frac 12gq<fb+\frac 32gq,
\end{align}
and the wave configuration is J-(R/S)-J-(R/S).
\item[(iv)] Now consider the state space
 \begin{align}\label{statespace_4}
    \mathcal{U}_4=\{(f, b, g, q)\in \mathbb{R}^4:  f, g, q, b>0, 3gq>fb>gq\}.
\end{align} 
The ordering is
  \begin{align}
    \frac 12fb<fb+\frac 12gq<\frac 32fb<fb+\frac 32gq.
\end{align}
The wave configuration is J-J-(R/S)-(R/S).
\item[(v)] Considering 
 \begin{align}\label{statespace_5}
    \mathcal{U}_5=\{(f, b, g, q)\in \mathbb{R}^4:   f, g, q, b>0, 3gq<fb\},
\end{align}
we have
  \begin{align}
    \frac 12fb<fb+\frac 12gq<fb+\frac 32gq<\frac 32fb.
\end{align}
Thus, the wave configuration is also J-J-(R/S)-(R/S). However, the third and fourth nonlinear waves change their roles.
\end{enumerate}

\section*{Acknowledgments}
This work was financially supported by the German Research Foundation (DFG), within the Priority Programme - SPP 2410 Hyperbolic Balance Laws in Fluid Mechanics: Complexity, Scales, Randomness (CoScaRa), the Sino-German Center on Advanced Numerical Methods for Nonlinear Hyperbolic Balance Laws and Their Applications (under the project number GZ1465),  the Natural Science Foundation of China (NSFC) (Nos. 12371391) and the Funding of National Key Laboratory of Computational Physics.
%
%
%
%
\bibliographystyle{plain}
\bibliography{references}
%
%

\end{document}